\pdfoutput=1

\documentclass{eptcs}

 \providecommand{\publicationstatus}{}

\usepackage{bookmark}

\usepackage[table]{xcolor}

\usepackage[all,cmtip]{xy} 

\usepackage{tikz}
\usepackage{circuitikz}

\usepackage{amsmath}
\usepackage{amssymb}

\usepackage{stmaryrd}

\usepackage{graphicx}
\usepackage{pdflscape}
\usepackage{geometry}
\usepackage{changepage}
\usepackage{mathtools}

\usepackage{mdframed}

\usepackage{comment}

\usepackage{arydshln}

\usepackage{multicol}

\usepackage{float}

\usepackage{everypage}
\usepackage{lipsum}

\usepackage{amsmath}
\usepackage{amsfonts}
\usepackage{amsthm}
\usepackage[inline]{enumitem}

\usepackage{scalerel,stackengine}
\stackMath
\renewcommand\hat[1]{%
\savestack{\tmpbox}{\stretchto{%
  \scaleto{%
    \scalerel*[\widthof{\ensuremath{#1}}]{\kern-.6pt\bigwedge\kern-.6pt}%
    {\rule[-\textheight/2]{1ex}{\textheight}}
  }{\textheight}%
}{0.5ex}}%
\stackon[1pt]{#1}{\tmpbox}%
}
\makeatletter
\newcommand{\inlineitem}[1][]{%
\ifnum\enit@type=\tw@
    {\descriptionlabel{#1}}
  \hspace{\labelsep}%
\else
  \ifnum\enit@type=\z@
       \refstepcounter{\@listctr}\fi
    \quad\@itemlabel\hspace{\labelsep}%
\fi}
\makeatother
\newcommand{\xrightarrowtail}[1]{\!\!{\xymatrix@C=1em{\ar@{>->}[r]^{#1}&}}\!\!\!}
\newcommand{\xleftarrowtail}[1]{\!\!\!{\xymatrix@C=1em{&\ar@{>->}[l]_{#1}}}\!\!}

\newcommand{\xrightarrowhook}[1]{\!\!{\xymatrix@C=1em{\ar@{^{(}->}[r]^{#1}&}}\!\!\!}
\newcommand{\xleftarrowhook}[1]{\!\!\!{\xymatrix@C=1em{&\ar@{_{(}->}[l]_{#1}}}\!\!}

\newcommand{\xrightarrowiso}[1]{\!\!{\xymatrix@C=1em{\ar@{->}[r]^{#1}_\cong&}}\!\!\!}
\newcommand{\xleftarrowiso}[1]{\!\!\!{\xymatrix@C=1em{&\ar@{->}[l]_{#1}^\cong}}\!\!}

\usetikzlibrary{shapes.geometric}
\usetikzlibrary{patterns}
\usetikzlibrary{fit}
\usetikzlibrary{positioning}
\usetikzlibrary{calc}
\usetikzlibrary{arrows}
\usetikzlibrary{decorations.markings}
\usetikzlibrary{decorations.pathreplacing}
\usetikzlibrary{shapes}

\usepackage{mathdots}

\pgfdeclarelayer{background}
\pgfdeclarelayer{nodelayer}
\pgfdeclarelayer{edgelayer}
\pgfsetlayers{background,edgelayer,nodelayer,main}

\tikzset{H/.style={draw,color=black,fill={rgb:black,1;white,3}, rectangle}}

\tikzset{rn/.style={}}
\tikzset{simple/.style={}}
\tikzset{none/.style={}}
\tikzset{nothing/.style={}}
\tikzset{thick/.style={draw, line width=0.5mm }}
\tikzset{Z/.style={draw,fill=white, circle,scale=1, inner sep=0pt, minimum size=10pt}}
\tikzset{X/.style={draw,fill={rgb:black,1;white,3}, text=black, circle,scale=1, inner sep=0pt, minimum size=10pt }}
\tikzset{Xthick/.style={draw,fill=white, circle,scale=1, inner sep=0pt, minimum size=15pt,line width=0.5mm}}
\tikzset{Zthick/.style={draw,fill={rgb:black,1;white,3}, text=black, circle,scale=1, inner sep=0pt, minimum size=15pt,line width=0.5mm }}

\tikzset{phase/.style={draw,fill=white, diamond,scale=1, inner sep=0pt, minimum size=10pt}}

\tikzset{discard/.style={draw, xscale=2.2,ground, rotate=90}}
\tikzset{mmixed/.style={draw, quantum, yscale=-2.2,ground, rotate=180}}
\tikzset{quantum/.style={line width=.6mm}}
\tikzset{map/.style={draw,color=black,fill=white, rectangle}}
\tikzset{s/.style={draw,color=black,fill=black, rectangle}}
\tikzset{mapthick/.style={draw,color=black,fill=white, rectangle, inner sep=0pt, minimum size=15pt,line width=0.5mm }}

\tikzset{otimes/.style={draw,fill=white,rotate=45, scale=0.9,minimum height=.1cm,circle,append after command={
[shorten >=\pgflinewidth, shorten <=\pgflinewidth,]
(\tikzlastnode.north) edge (\tikzlastnode.south)
(\tikzlastnode.east) edge (\tikzlastnode.west)
}
}
}

\tikzset{dot/.style={thick, fill=black, circle, scale=1, inner sep = .05cm}}

\tikzset{oplus/.style={draw, scale=0.9,minimum height=.1cm,circle,append after command={
[shorten >=\pgflinewidth, shorten <=\pgflinewidth,]
(\tikzlastnode.north) edge (\tikzlastnode.south)
(\tikzlastnode.east) edge (\tikzlastnode.west)
}
}
}

\usetikzlibrary{arrows,shapes.gates.logic.US,shapes.gates.logic.IEC,calc}

\tikzset{andin/.style={
draw,
and gate US,
rotate=90,
scale=1,
fill=white,
label={center:{\it \&}}
}}

\tikzset{mulin/.style={
draw,
and gate US,
rotate=90,
scale=1,
fill=white,
label={center:{}}
}}

\tikzset{andout/.style={
draw,
and gate US,
rotate=-90,
scale=1,
fill=white,
label={center:{\it \&}}
}}

\tikzset{tri/.style={
draw,
shape border rotate=-30,
regular polygon,
regular polygon sides=3,
fill={rgb:black,1;white,3},
inner sep = .1cm
}
}

\tikzset{triflip/.style={
draw,
shape border rotate=0,
regular polygon,
regular polygon sides=3,
fill={rgb:black,1;white,3},
inner sep = .1cm
}
}

\tikzset{fanin/.style={
draw,
shape border rotate=30,
regular polygon,
regular polygon sides=3,
fill=white,
inner sep = .1cm
}
}

\tikzset{fanout/.style={
draw,
shape border rotate=-30,
regular polygon,
regular polygon sides=3,
fill=white,
inner sep = .1cm
}
}
\tikzstyle{strings}=[baseline={([yshift=-.5ex]current bounding box.center)}]

\tikzset{every picture/.append style={scale=.55}, transform shape,strings}

  \newtheorem{theorem}{Theorem}[section]
  
  \newtheorem{lemma}[theorem]{Lemma}

  \newtheorem{definition}[theorem]{Definition}
  
  \newtheorem{conjecture}[theorem]{Conjecture}
  \newtheorem{remark}[theorem]{Remark}

\newcommand{\Mat}{\mathsf{Mat}}

\newcommand{\FSets}{\mathsf{FinSet}}
\newcommand{\FinOrd}{\mathsf{FinSet}}

\newcommand{\Span}{\mathsf{Span}}

\newcommand{\op}{\mathsf{op}}

\newcommand{\FPinj}{\mathsf{FPinj}}
\newcommand{\FPar}{\mathsf{FPar}}
\newcommand{\FSpan}{\mathsf{FSpan}}

\newcommand{\Par}{\mathsf{Par}}
\newcommand{\Aff}{\mathsf{Aff}}
\newcommand{\ParIso}{\mathsf{ParIso}}

\newcommand{\Fin}{\mathsf{Fd}}

\renewcommand{\sp}{\mathsf{sp}}
\newcommand{\pr}{\mathsf{p}}
\newcommand{\iso}{\mathsf{i}}

\newcommand{\Csp}{{\sf Cospan}}

\newcommand{\Iso}{{\sf Iso}}
\renewcommand{\P}{{\sf p}}

\newcommand{\Prof}{{\sf Prof}}

\newcommand{\inj}{{\sf Mono}}
\newcommand{\surj}{{\sf Epi}}

\newcommand{\sub}{{\sf sub}}

\newcommand{\cm}{{\sf cm}}
\newcommand{\cb}{{\sf cb}}

\newcommand{\Mon}{{\sf Mon}}

\newcommand{\F}{\mathbb{F}}

\newcommand{\f}{\mathsf{f}}
\newcommand{\X}{\mathbb{X}}
\newcommand{\Y}{\mathbb{Y}}
\newcommand{\Z}{\mathbb{Z}}
\newcommand{\N}{\mathbb{N}}

\newcommand{\dom}{{\sf dom}}
\newcommand{\cod}{{\sf cod}}

\newcommand{\ZXA}{\mathsf{ZX}\textit{\&}}

\newcommand{\Vect}{\mathsf{FVect}}

\DeclareMathSymbol{\bot}{\mathord}{symbols}{"3F}

\renewcommand{\epsilon}{\varepsilon}
\renewcommand{\bar}[1]{\overline{#1}\hspace*{.01cm}}

\usepackage{amsmath}
\usepackage{pict2e}

\newcommand{\lbparen}{\{
}

\newcommand{\rbparen}{ \}
}

\newdir{|>}{-<5pt,0pt>{
\begin{tikzpicture}[scale=.7]
	\begin{pgfonlayer}{nodelayer}
		\node [style=none] (0) at (0, 0) {};
		\node [style=none] (1) at (1, 0) {};
		\node [style=none] (2) at (-1, -0.25) {};
	\end{pgfonlayer}
	\begin{pgfonlayer}{edgelayer}
		\draw (2.center) to (0.center);
		\draw (0.center) to (1.center);
	\end{pgfonlayer}
\end{tikzpicture}
}}
\newdir{|<}{-<5pt,0pt>{
\begin{tikzpicture}[scale=.9]
	\begin{pgfonlayer}{nodelayer}
		\node [style=none] (0) at (0, -0.25) {};
		\node [style=none] (1) at (-1, -0.25) {};
		\node [style=none] (2) at (1, 0) {};
	\end{pgfonlayer}
	\begin{pgfonlayer}{edgelayer}
		\draw (2.center) to (0.center);
		\draw (0.center) to (1.center);
	\end{pgfonlayer}
\end{tikzpicture}
}}

\newcommand{\skewpullbackcorner}[1][dl]{\save*!/#1-1.1pc/#1:(-.5,1)@^{|>}\restore}

\newcommand{\ev}{{\sf ev}}

\usepackage{amsmath}

\usepackage{hyperref}

\newcommand\numeq[2]%
  {\label{#2}\stackrel{\scriptscriptstyle(\mkern-1.5mu#1\mkern-1.5mu)}{=}}

\xymatrixrowsep{.5cm}
\xymatrixcolsep{.65cm}

\title{Distributive Laws, Spans and the ZX-Calculus}
\newcommand{\titlerunning}{Distributive Laws, Spans and the ZX-calculus}
\date{\today}
\author{Cole Comfort\\ Department of Computer Science, University of Oxford}
\def\authorrunning{Cole Comfort}

\newcounter{eq}

\makeatletter
\newcommand{\ltxlabel}{\ltx@label}
\makeatother

\newcommand{\eqstack}[1]{%
\stackrel{\scalebox{.6}{(\ref{#1})}}{=}%
}

\newcommand{\eq}[1]{%
\refstepcounter{eq}%
\ltxlabel{#1}%
\eqstack{#1}%
}

\newcommand{\eref}{\eqstack}

\newcommand{\erefop}[1]{%
\stackrel{\scalebox{.6}{$(\ref{#1})^\op$}}{=}%
}

\begin{document}

\maketitle

\newcommand{\cubetopbl}{A}
\newcommand{\cubetopbr}{B}
\newcommand{\cubetopfl}{C}
\newcommand{\cubetopfr}{D}
\newcommand{\cubebotbl}{E}
\newcommand{\cubebotbr}{F}
\newcommand{\cubebotfl}{G}
\newcommand{\cubebotfr}{H}

%

\begin{abstract}
We modularly build increasingly larger fragments of the ZX-calculus by modularly adding new generators and relations, at each point, giving some concrete semantics in terms of some category of spans.  This is performed using Lack's technique of composing props via distributive laws \cite{lack}, as well as the technique of pushout cubes of Zanasi \cite{zanasi}.
We do this for the fragment of the ZX-calculus with only the black $\pi$-phase (and no Hadamard gate) as well as well as the fragment which additionally has the and gate as a generator (which is equivalent to the natural number H-box fragment of the ZH-calculus).
In the former case, we show that this is equivalent to the full subcategory of spans of (possibly empty) free, finite dimensional affine $\F_2$-vector spaces, where the objects are the non-empty affine vector spaces.  In the latter case, we show that this is equivalent to the full subcategory of spans of finite sets where the objects are powers of the two element set.  Because these fragments of the ZX-calculus have semantics in terms of {\em full subcategories} of categories of spans, they can not be presented by distributive laws over groupoids.  Instead, we first construct their subcategories of partial isomorphisms via distributive laws over all isomorphims with subobjects adjoined.  After which, the full subcategory of spans are obtained by freely adjoining units and counits the the semi-Frobenius structures given by the diagonal and codiagonal maps.

\end{abstract}

We first review some basic theory involving the presentation of props.  These results are mostly folklore, however, I will refer the reader to \cite[\S 2]{zanasi} for a more comprehensive introduction.

\begin{definition}
A {\bf pro} is a strict monoidal category generated by one object under the tensor product, and a {\bf prop} is a  strict {\em symmetric} monoidal category generated by one object under the tensor product

\end{definition}

\begin{definition}
A {\bf monoidal theory} is a pair $(\Sigma,E)$ of {\bf generators} $\Sigma$ and {\bf equations} $E$.
Each generator $f \in \Sigma$ has a chosen domain $\dom (f) \in \N$  and codomain $\cod (f) \in \N$, so that $f$ can be seen as a map from $\dom(f)$ to $\cod (f)$.

The free pro with signature $\Sigma$ has maps in $\Sigma^*$ obtained by inductively  tensoring all the generators and composing all appropriately typed generators in $\Sigma$,
The equations in $E$ are pairs of parallel maps in $\Sigma^*$.
Any monoidal theory $(\Sigma,E)$  generates a pro $\bar{(\Sigma,E)}$ given by the free pro with signature $\Sigma$ modulo the equations in $E$.

A {\bf symmetric monoidal theory} is the symmetric version of a monoidal theory, which generates a prop.  Here the set $\Sigma^*$ is obtained by composing and tensoring maps with symmetries, and then quotienting by the axioms of a prop.
\end{definition}

\begin{lemma}
Given two (symmetric) monoidal theories $(\Sigma_1,E_1)$  and $(\Sigma_2,E_2)$  the coproduct of pro(p)s  $\bar{(\Sigma_1,E_1)}+\bar{(\Sigma_2,E_2)}$ is generated by the (symmetric) monoidal theory $(\Sigma_1+\Sigma_2,E_1+E_2)$.
\end{lemma}

\begin{lemma}
Given three  (symmetric) monoidal theories $(\Sigma_1,E_1)$, $(\Sigma_2,E_2)$ and $(\Sigma_3,E_3)$ where $\bar{(\Sigma_3,E_3)}$ is a sub-pro(p) of both $\bar{(\Sigma_1,E_1)}$ and $\bar{(\Sigma_2,E_2)}$.  The pushout of the diagram of pro(p)s
$$
\bar{(\Sigma_1,E_1)} \leftarrow \bar{(\Sigma_3,E_3)}\rightarrow \bar{(\Sigma_2,E_2)}
$$
is generated by the (symmetric) monoidal theory $(\Sigma_1^* +_{\Sigma_3} \Sigma_2^*, E_1 + E_2)$.
\end{lemma}
%
%
%

We recall the novel way to compose pro(p), first described in \cite{lack}:
\begin{definition}
Suppose there three  (symmetric) monoidal theories $(\Sigma_1,E_1)$, $(\Sigma_2,E_2)$ and $(\Sigma_3,E_3)$ where $\bar{(\Sigma_3,E_3)}$ is a sub-pro(p) of both $\bar{(\Sigma_1,E_1)}$ and $\bar{(\Sigma_2,E_2)}$. A {\bf distributive law of pro(p)s} is a distributive law $\lambda:\bar{(\Sigma_2,E_2)} \otimes_{\bar{(\Sigma_3,E_3)}} \bar{ (\Sigma_1,E_1)}$   in $\Mon$-$\Prof$.  Informally, this is a way to push all the maps in $\Sigma_1^*$ past those of  $\Sigma_2^*$ modulo $\Sigma_3$ and the equations $E_1+E_2$ and the axioms of a pro(p).
\end{definition}

In \cite{lack} it is required that $\bar{(\Sigma_3,E_3)}$ is a groupoid; however, we must loosen this requirement (note that when this is not a groupoid, there is no correspondence to factorization systems as in \cite{rosebrugh}).  

\begin{lemma}
Suppose that we have three (symmetric) monoidal theories and a distribtuive law $\lambda:\bar{(\Sigma_2,E_2)} \otimes_{\bar{(\Sigma_3,E_3)}} \bar{ (\Sigma_1,E_1)}$ as above.

Then the induced pro(p) $\bar{(\Sigma_2,E_2)} \otimes_{\bar{(\Sigma_3,E_3)}} \bar{ (\Sigma_1,E_1)}$ is presented by the monoidal theory\\ ${(\Sigma_1^* +_{\Sigma_3} \Sigma_2^*, E_1 + E_2+E_\lambda)}$, where $E_\lambda$ are all the equations needed to push elements of $\Sigma_1^*$ past those of $ \Sigma_2^*$ up to  $\Sigma_3^*$, dictated by $\lambda$.
\end{lemma}

There is a folklore result that the prop for the free commutative monoid is equivalent to the category of finite sets and functions under the direct sum.  The string diagrams correspond to drawing the ``graphs'' of functions.

\begin{definition}
Let $\inj$ be the free prop with one generator of type $0\to 1$  drawn as a black circle; and $\surj$ be the prop generated by the free associative commutative binary operationm graphically generated by the following\footnote{Diagrams are read from bottom to top.}:
  $$
  \begin{tikzpicture}[rotate=90]
	\begin{pgfonlayer}{nodelayer}
		\node [style=X] (0) at (-9, -0) {};
		\node [style=none] (1) at (-8.25, -0) {};
		\node [style=X] (2) at (-9.75, 0.25) {};
		\node [style=none] (3) at (-9.75, -0.25) {};
		\node [style=none] (4) at (-10.5, 0.5) {};
		\node [style=none] (5) at (-10.5, -0) {};
		\node [style=none] (6) at (-10.5, -0.25) {};
	\end{pgfonlayer}
	\begin{pgfonlayer}{edgelayer}
		\draw [in=-150, out=0, looseness=1.00] (3.center) to (0);
		\draw [in=150, out=0, looseness=1.00] (2.center) to (0);
		\draw (0) to (1.center);
		\draw (3.center) to (6.center);
		\draw [in=-150, out=0, looseness=1.00] (5.center) to (2);
		\draw [in=0, out=150, looseness=1.00] (2) to (4.center);
	\end{pgfonlayer}
\end{tikzpicture}
  \eq{assoc}
  \begin{tikzpicture}[rotate=90,yscale=-1]
	\begin{pgfonlayer}{nodelayer}
		\node [style=X] (0) at (-9, -0) {};
		\node [style=none] (1) at (-8.25, -0) {};
		\node [style=X] (2) at (-9.75, 0.25) {};
		\node [style=none] (3) at (-9.75, -0.25) {};
		\node [style=none] (4) at (-10.5, 0.5) {};
		\node [style=none] (5) at (-10.5, -0) {};
		\node [style=none] (6) at (-10.5, -0.25) {};
	\end{pgfonlayer}
	\begin{pgfonlayer}{edgelayer}
		\draw [in=-150, out=0, looseness=1.00] (3.center) to (0);
		\draw [in=150, out=0, looseness=1.00] (2.center) to (0);
		\draw (0) to (1.center);
		\draw (3.center) to (6.center);
		\draw [in=-150, out=0, looseness=1.00] (5.center) to (2);
		\draw [in=0, out=150, looseness=1.00] (2) to (4.center);
	\end{pgfonlayer}
\end{tikzpicture},
  \hspace*{1cm}
  \begin{tikzpicture}[rotate=90]
	\begin{pgfonlayer}{nodelayer}
		\node [style=X] (0) at (-9, -0) {};
		\node [style=none] (1) at (-8.25, -0) {};
		\node [style=none] (2) at (-9.75, 0.25) {};
		\node [style=none] (3) at (-9.75, -0.25) {};
		\node [style=none] (4) at (-10.5, 0.25) {};
		\node [style=none] (5) at (-10.5, -0.25) {};
	\end{pgfonlayer}
	\begin{pgfonlayer}{edgelayer}
		\draw [in=-150, out=0, looseness=1.00] (3.center) to (0);
		\draw [in=150, out=0, looseness=1.00] (2.center) to (0);
		\draw (0) to (1.center);
		\draw [in=180, out=0, looseness=1.00] (5.center) to (2.center);
		\draw [in=0, out=180, looseness=1.00] (3.center) to (4.center);
	\end{pgfonlayer}
  \end{tikzpicture}
  \eq{comm}
  \begin{tikzpicture}[rotate=90]
	\begin{pgfonlayer}{nodelayer}
		\node [style=X] (0) at (-9, -0) {};
		\node [style=none] (1) at (-8.25, -0) {};
		\node [style=none] (2) at (-9.75, 0.25) {};
		\node [style=none] (3) at (-9.75, -0.25) {};
	\end{pgfonlayer}
	\begin{pgfonlayer}{edgelayer}
		\draw [in=-150, out=0, looseness=1.00] (3.center) to (0);
		\draw [in=150, out=0, looseness=1.00] (2.center) to (0);
		\draw (0) to (1.center);
	\end{pgfonlayer}
\end{tikzpicture}
  $$
\end{definition}

\begin{definition}
Let $\inj(\X)$ and $\surj(\X)$ denote the subcategories of monomorphisms and epimorphisms of $\X$.
\end{definition}

\begin{lemma}
The category $(\inj(\FSets),+)$ is presented by the prop $\inj$ and $(\surj(\FSets),+)$ by the prop $\surj$.
Moreover, $(\FSets,+)$ is presented by the distributive law:
$$
\surj \otimes_\P \inj;
  \begin{tikzpicture}[rotate=90]
	\begin{pgfonlayer}{nodelayer}
		\node [style=X] (0) at (-9, -0) {};
		\node [style=none] (1) at (-8.25, -0) {};
		\node [style=X] (2) at (-9.75, 0.25) {};
		\node [style=none] (3) at (-10, -0.25) {};
	\end{pgfonlayer}
	\begin{pgfonlayer}{edgelayer}
		\draw [in=-150, out=0, looseness=1.00] (3.center) to (0);
		\draw [in=150, out=0, looseness=1.00] (2.center) to (0);
		\draw (0) to (1.center);
	\end{pgfonlayer}
  \end{tikzpicture}
=
  \begin{tikzpicture}[rotate=90,yscale=-1]
	\begin{pgfonlayer}{nodelayer}
		\node [style=X] (0) at (-9, -0) {};
		\node [style=none] (1) at (-8.25, -0) {};
		\node [style=X] (2) at (-9.75, 0.25) {};
		\node [style=none] (3) at (-10, -0.25) {};
	\end{pgfonlayer}
	\begin{pgfonlayer}{edgelayer}
		\draw [in=-150, out=0, looseness=1.00] (3.center) to (0);
		\draw [in=150, out=0, looseness=1.00] (2.center) to (0);
		\draw (0) to (1.center);
	\end{pgfonlayer}
  \end{tikzpicture}
  \eq{unit}
  \begin{tikzpicture}[rotate=90]
	\begin{pgfonlayer}{nodelayer}
		\node [style=none] (0) at (-9, 0.25) {};
		\node [style=none] (1) at (-9.75, 0.25) {};
	\end{pgfonlayer}
	\begin{pgfonlayer}{edgelayer}
		\draw (1) to (0.center);
	\end{pgfonlayer}
  \end{tikzpicture}
$$
Yielding the prop for the free commutative monoid, $\cm$.

\end{lemma}

As observed in \cite[Ex. 3.3 (a)]{bonchi}, this distributive law corresponds to the epi-mono factorization of functions: one can always decompose a function into surjections following by injections by applying this equation multiple times.

Lack showed that the structure of special commutative Frobenius algebras and bicommutative bialgebras arise in terms of distributive laws corresponding to pushout and pullback in $(\FSets,+)$  \cite[\S 5.3, 5.4]{lack}:

\begin{definition}
Consider the following two distributive laws: 
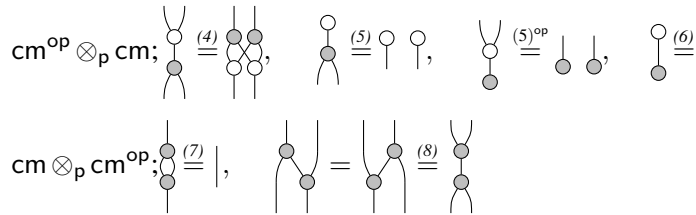
\begin{align*}
\cm^\op  \otimes_\P \cm;&
  \begin{tikzpicture}
	\begin{pgfonlayer}{nodelayer}
		\node [style=X] (0) at (-3.75, -1) {};
		\node [style=none] (1) at (-4, -1.75) {};
		\node [style=none] (2) at (-3.5, -1.75) {};
		\node [style=Z] (3) at (-3.75, -0.25) {};
		\node [style=none] (4) at (-4, 0.5) {};
		\node [style=none] (5) at (-3.5, 0.5) {};
	\end{pgfonlayer}
	\begin{pgfonlayer}{edgelayer}
		\draw [in=90, out=-60, looseness=1.00] (0) to (2.center);
		\draw [in=-120, out=90, looseness=1.00] (1.center) to (0);
		\draw (0) to (3);
		\draw [in=60, out=-90, looseness=1.00] (5.center) to (3);
		\draw [in=-90, out=120, looseness=1.00] (3) to (4.center);
	\end{pgfonlayer}
  \end{tikzpicture}
  \eq{bi.one}
  \begin{tikzpicture}
	\begin{pgfonlayer}{nodelayer}
		\node [style=X] (0) at (-4, 0.5) {};
		\node [style=Z] (1) at (-4, -0.25) {};
		\node [style=X] (2) at (-4.5, 0.5) {};
		\node [style=Z] (3) at (-4.5, -0.25) {};
		\node [style=none] (4) at (-4, -1) {};
		\node [style=none] (5) at (-4.5, -1) {};
		\node [style=none] (6) at (-4.5, 1.25) {};
		\node [style=none] (7) at (-4, 1.25) {};
	\end{pgfonlayer}
	\begin{pgfonlayer}{edgelayer}
		\draw [bend left, looseness=1.25] (0) to (1);
		\draw [bend right, looseness=1.25] (2) to (3);
		\draw (1) to (2);
		\draw (3) to (0);
		\draw (0) to (7.center);
		\draw (6.center) to (2);
		\draw (3) to (5.center);
		\draw (4.center) to (1);
	\end{pgfonlayer}
\end{tikzpicture},
\hspace*{.5cm}
  \begin{tikzpicture}
	\begin{pgfonlayer}{nodelayer}
		\node [style=Z] (0) at (-4, -0) {};
		\node [style=X] (1) at (-4, -0.75) {};
		\node [style=none] (2) at (-4.25, -1.5) {};
		\node [style=none] (3) at (-3.75, -1.5) {};
	\end{pgfonlayer}
	\begin{pgfonlayer}{edgelayer}
		\draw [in=-60, out=90, looseness=1.00] (3.center) to (1);
		\draw (1) to (0);
		\draw [in=90, out=-120, looseness=1.00] (1) to (2.center);
	\end{pgfonlayer}
  \end{tikzpicture}
  \eq{bi.two}
  \begin{tikzpicture}
	\begin{pgfonlayer}{nodelayer}
		\node [style=Z] (0) at (-4.25, -0.75) {};
		\node [style=none] (1) at (-4.25, -1.5) {};
		\node [style=none] (2) at (-3.5, -1.5) {};
		\node [style=Z] (3) at (-3.5, -0.75) {};
	\end{pgfonlayer}
	\begin{pgfonlayer}{edgelayer}
		\draw (2.center) to (3);
		\draw (0) to (1.center);
	\end{pgfonlayer}
  \end{tikzpicture},
  \hspace*{.5cm}
   \begin{tikzpicture}[yscale=-1]
	\begin{pgfonlayer}{nodelayer}
		\node [style=X] (0) at (-4, -0) {};
		\node [style=Z] (1) at (-4, -0.75) {};
		\node [style=none] (2) at (-4.25, -1.5) {};
		\node [style=none] (3) at (-3.75, -1.5) {};
	\end{pgfonlayer}
	\begin{pgfonlayer}{edgelayer}
		\draw [in=-60, out=90, looseness=1.00] (3.center) to (1);
		\draw (1) to (0);
		\draw [in=90, out=-120, looseness=1.00] (1) to (2.center);
	\end{pgfonlayer}
  \end{tikzpicture}
  \erefop{bi.two}
   \begin{tikzpicture}[yscale=-1]
	\begin{pgfonlayer}{nodelayer}
		\node [style=X] (0) at (-4.25, -0.75) {};
		\node [style=none] (1) at (-4.25, -1.5) {};
		\node [style=none] (2) at (-3.5, -1.5) {};
		\node [style=X] (3) at (-3.5, -0.75) {};
	\end{pgfonlayer}
	\begin{pgfonlayer}{edgelayer}
		\draw (2.center) to (3);
		\draw (0) to (1.center);
	\end{pgfonlayer}
  \end{tikzpicture},
\hspace*{.5cm}
  \begin{tikzpicture}[rotate=90]
	\begin{pgfonlayer}{nodelayer}
		\node [style=Z] (0) at (-8.25, -0) {};
		\node [style=X] (1) at (-9.25, -0) {};
	\end{pgfonlayer}
	\begin{pgfonlayer}{edgelayer}
		\draw (0) to (1);
	\end{pgfonlayer}
\end{tikzpicture}
  \eq{extra}
\\
 \cm \otimes_\P \cm^\op;&
    \begin{tikzpicture}[rotate=90]
	\begin{pgfonlayer}{nodelayer}
		\node [style=X] (0) at (-6.25, 0.25) {};
		\node [style=none] (1) at (-7, 0.25) {};
		\node [style=none] (2) at (-4.75, 0.25) {};
		\node [style=X] (3) at (-5.5, 0.25) {};
	\end{pgfonlayer}
	\begin{pgfonlayer}{edgelayer}
		\draw (0) to (1.center);
		\draw (3) to (2.center);
		\draw [bend right, looseness=1.25] (3) to (0);
		\draw [bend right, looseness=1.25] (0) to (3);
	\end{pgfonlayer}
  \end{tikzpicture}
  \eq{special}
  \begin{tikzpicture}[rotate=90]
	\begin{pgfonlayer}{nodelayer}
		\node [style=none] (0) at (-7, 0.25) {};
		\node [style=none] (1) at (-6, 0.25) {};
	\end{pgfonlayer}
	\begin{pgfonlayer}{edgelayer}
		\draw (1.center) to (0.center);
	\end{pgfonlayer}
  \end{tikzpicture},
  \hspace*{.5cm}
  \begin{tikzpicture}[rotate=90]
	\begin{pgfonlayer}{nodelayer}
		\node [style=X] (0) at (-7, -0) {};
		\node [style=X] (1) at (-6.25, 0.5) {};
		\node [style=none] (2) at (-7, 0.75) {};
		\node [style=none] (3) at (-7.75, 0.75) {};
		\node [style=none] (4) at (-7.75, -0) {};
		\node [style=none] (5) at (-6.25, -0.25) {};
		\node [style=none] (6) at (-5.5, -0.25) {};
		\node [style=none] (7) at (-5.5, 0.5) {};
	\end{pgfonlayer}
	\begin{pgfonlayer}{edgelayer}
		\draw (6.center) to (5.center);
		\draw [in=-30, out=180, looseness=1.00] (5.center) to (0);
		\draw (1) to (0);
		\draw [in=0, out=150, looseness=1.00] (1) to (2.center);
		\draw (2.center) to (3.center);
		\draw (0) to (4.center);
		\draw (1) to (7.center);
	\end{pgfonlayer}
  \end{tikzpicture}
 =
  \begin{tikzpicture}[rotate=90,xscale=-1]
	\begin{pgfonlayer}{nodelayer}
		\node [style=X] (0) at (-7, -0) {};
		\node [style=X] (1) at (-6.25, 0.5) {};
		\node [style=none] (2) at (-7, 0.75) {};
		\node [style=none] (3) at (-7.75, 0.75) {};
		\node [style=none] (4) at (-7.75, -0) {};
		\node [style=none] (5) at (-6.25, -0.25) {};
		\node [style=none] (6) at (-5.5, -0.25) {};
		\node [style=none] (7) at (-5.5, 0.5) {};
	\end{pgfonlayer}
	\begin{pgfonlayer}{edgelayer}
		\draw (6.center) to (5.center);
		\draw [in=-30, out=180, looseness=1.00] (5.center) to (0);
		\draw (1) to (0);
		\draw [in=0, out=150, looseness=1.00] (1) to (2.center);
		\draw (2.center) to (3.center);
		\draw (0) to (4.center);
		\draw (1) to (7.center);
	\end{pgfonlayer}
  \end{tikzpicture}
  \eq{frob}
  \begin{tikzpicture}[rotate=90]
	\begin{pgfonlayer}{nodelayer}
		\node [style=none] (0) at (-4.75, -0.25) {};
		\node [style=X] (1) at (-5.5, -0) {};
		\node [style=none] (2) at (-7, -0.25) {};
		\node [style=X] (3) at (-6.25, 0) {};
		\node [style=none] (4) at (-4.75, 0.25) {};
		\node [style=none] (5) at (-7, 0.25) {};
	\end{pgfonlayer}
	\begin{pgfonlayer}{edgelayer}
		\draw [in=-30, out=180, looseness=1.25] (0.center) to (1);
		\draw (3) to (1);
		\draw [in=180, out=30, looseness=1.25] (1) to (4.center);
		\draw [in=0, out=-150, looseness=1.25] (3) to (2.center);
		\draw [in=0, out=150, looseness=1.25] (3) to (5.center);
	\end{pgfonlayer}
\end{tikzpicture}
  \end{align*}

The former yields, {\sf cb}, the prop for the free {\bf bicommutative bialgebra} and the latter yields, {\sf scfa}, the prop for the free {\bf special commutative Frobenius algebra}.

\end{definition}

\begin{lemma} \cite[\S 5.3, 5.4]{lack}
{\sf cb} is a presentation for $(\Span(\FSets),+)$ and {\sf scfa} is a presentation for $(\Csp(\FSets),+)$.

\end{lemma}

The equations generating these distributive laws give us a recipe for how to generate all pushouts and pullbacks along epics and monics.
This category of spans can be seen from a slightly different perspective:

\begin{lemma}
There is an equivalence of props $\cb\cong(\Span(\FinOrd),+)\cong (\Mat(\N),+)$.
\end{lemma}
Although we are quite sure that the second equivalence is also folklore, a similar result is given in \cite{bruni}.

One way to see this is by interpreting the monoid as addition and 0 and the comonoid as copying/deleting.  For example, consider the interpretation of the following diagram in $\cm^\op \otimes_\P \cm$ as a matrix:
\hfil
$
\left\llbracket
\begin{tikzpicture}
	\begin{pgfonlayer}{nodelayer}
		\node [style=none] (0) at (-0.5, -1) {};
		\node [style=none] (1) at (0.75, -1) {};
		\node [style=none] (2) at (1.5, -1) {};
		\node [style=Z] (4) at (0.75, -0.5) {};
		\node [style=Z] (5) at (1.5, -0.5) {};
		\node [style=Z] (6) at (0.25, 0.25) {};
		\node [style=X] (7) at (1, 1) {};
		\node [style=X] (8) at (-0.5, 1) {};
		\node [style=none] (9) at (1, 1.5) {};
		\node [style=none] (10) at (-0.5, 1.5) {};
		\node [style=X] (14) at (0.25, 1) {};
		\node [style=none] (15) at (0.25, 1.5) {};
	\end{pgfonlayer}
	\begin{pgfonlayer}{edgelayer}
		\draw [in=-30, out=150] (6) to (8);
		\draw [in=135, out=-90] (6) to (4);
		\draw [bend right, looseness=1.25] (4) to (7);
		\draw [in=-150, out=30] (6) to (7);
		\draw [in=-135, out=90, looseness=0.75] (0.center) to (8);
		\draw (4) to (1.center);
		\draw (2.center) to (5);
		\draw (8) to (10.center);
		\draw (7) to (9.center);
		\draw (14) to (15.center);
	\end{pgfonlayer}
\end{tikzpicture}
\right\rrbracket
=
\begin{bmatrix}
  1 & 1 & 0\\
  0 & 0 & 0\\
  0 & 2 & 0 
\end{bmatrix}
$

\begin{definition}
Let $\cb_2$ denote the quotient of $\cb_2$ by the equation:
\hspace*{1cm}
$
\begin{tikzpicture}
	\begin{pgfonlayer}{nodelayer}
		\node [style=Z] (0) at (0, 0) {};
		\node [style=none] (1) at (0, -0.5) {};
		\node [style=X] (2) at (0, 0.75) {};
		\node [style=none] (3) at (0, 1.25) {};
	\end{pgfonlayer}
	\begin{pgfonlayer}{edgelayer}
		\draw (3.center) to (2);
		\draw (0) to (1.center);
		\draw [bend left=45, looseness=1.25] (0) to (2);
		\draw [bend left=45, looseness=1.25] (2) to (0);
	\end{pgfonlayer}
\end{tikzpicture}
\eq{hopf}
\begin{tikzpicture}
	\begin{pgfonlayer}{nodelayer}
		\node [style=Z] (0) at (0, 0) {};
		\node [style=none] (1) at (0, -0.5) {};
		\node [style=X] (2) at (0, 0.75) {};
		\node [style=none] (3) at (0, 1.25) {};
	\end{pgfonlayer}
	\begin{pgfonlayer}{edgelayer}
		\draw (3.center) to (2);
		\draw (0) to (1.center);
	\end{pgfonlayer}
\end{tikzpicture}
$
\end{definition}

\begin{lemma}
$\cb_2$ is a presentation for the prop $(\Mat(\F_2),+)$.
\end{lemma}

\begin{proof}
As an Abelian group $\F_2 \cong \Z/2\Z$; which is generated by the equation $2 \equiv 0$, corresponding to this quotient of $\cb$.
\end{proof}

\section{The phase-free fragment}
\label{sec:one}

In this section we build up to giving a presentation for $(\Span(\Mat(\F_2)),+)$ in a modular way. This category is shown to be the same as the phase-free Hadamard free fragment of the ZX-calculus. Although this presentation of linear spans has already been discussed in great detail  for arbitrary PIDs \cite{zanasi}, our particular method of exposition is necessary to motivate the affine and full cases.

\begin{definition}
Consider the prop $\Iso(\cb_2)$ generated by the controlled not gate:
\begin{align*}
&\left\llbracket
\begin{tikzpicture}
	\begin{pgfonlayer}{nodelayer}
		\node [style=oplus] (5) at (0.5, 2.75) {};
		\node [style=dot] (6) at (0, 2.75) {};
		\node [style=none] (7) at (0.5, 3.5) {};
		\node [style=none] (8) at (0.5, 2) {};
		\node [style=none] (9) at (0, 2) {};
		\node [style=none] (10) at (0, 3.5) {};
	\end{pgfonlayer}
	\begin{pgfonlayer}{edgelayer}
		\draw (8.center) to (5);
		\draw (5) to (7.center);
		\draw (10.center) to (6);
		\draw (6) to (5);
		\draw (6) to (9.center);
	\end{pgfonlayer}
\end{tikzpicture}
\right\rrbracket
=
\begin{tikzpicture}
	\begin{pgfonlayer}{nodelayer}
		\node [style=X] (0) at (-0.25, -1) {};
		\node [style=Z] (1) at (-0.75, -1.75) {};
		\node [style=none] (2) at (0, -2.25) {};
		\node [style=none] (3) at (-0.25, -0.5) {};
		\node [style=none] (4) at (-1, -0.5) {};
		\node [style=none] (5) at (-0.75, -2.25) {};
	\end{pgfonlayer}
	\begin{pgfonlayer}{edgelayer}
		\draw (3.center) to (0);
		\draw [in=90, out=-75] (0) to (2.center);
		\draw (5.center) to (1);
		\draw (1) to (0);
		\draw [in=-90, out=105] (1) to (4.center);
	\end{pgfonlayer}
\end{tikzpicture} 
\hspace*{1cm}
\text{modulo the following relations:}\\
\begin{tikzpicture}
	\begin{pgfonlayer}{nodelayer}
		\node [style=dot] (0) at (7, 0) {};
		\node [style=oplus] (1) at (7.5, 0) {};
		\node [style=oplus] (2) at (7.5, 0.5) {};
		\node [style=dot] (3) at (7, 0.5) {};
		\node [style=none] (4) at (7.5, 1) {};
		\node [style=none] (5) at (7, 1) {};
		\node [style=none] (6) at (7, -0.5) {};
		\node [style=none] (7) at (7.5, -0.5) {};
	\end{pgfonlayer}
	\begin{pgfonlayer}{edgelayer}
		\draw (5.center) to (3);
		\draw (3) to (0);
		\draw (0) to (6.center);
		\draw (7.center) to (1);
		\draw (1) to (2);
		\draw (2) to (4.center);
		\draw (2) to (3);
		\draw (0) to (1);
	\end{pgfonlayer}
\end{tikzpicture}
\eq{cnot.one}
\begin{tikzpicture}
	\begin{pgfonlayer}{nodelayer}
		\node [style=none] (4) at (7.5, 1) {};
		\node [style=none] (5) at (7, 1) {};
		\node [style=none] (6) at (7, -0.5) {};
		\node [style=none] (7) at (7.5, -0.5) {};
	\end{pgfonlayer}
	\begin{pgfonlayer}{edgelayer}
		\draw (7.center) to (4.center);
		\draw (5.center) to (6.center);
	\end{pgfonlayer}
\end{tikzpicture}
&\hspace*{.5cm}
\begin{tikzpicture}
	\begin{pgfonlayer}{nodelayer}
		\node [style=none] (4) at (7.5, 1) {};
		\node [style=none] (5) at (7, 1) {};
		\node [style=none] (6) at (7, -1) {};
		\node [style=none] (7) at (7.5, -1) {};
		\node [style=dot] (8) at (7.5, 0.5) {};
		\node [style=dot] (9) at (7.5, -0.5) {};
		\node [style=dot] (10) at (7, 0) {};
		\node [style=oplus] (11) at (7.5, 0) {};
		\node [style=oplus] (12) at (7, 0.5) {};
		\node [style=oplus] (13) at (7, -0.5) {};
	\end{pgfonlayer}
	\begin{pgfonlayer}{edgelayer}
		\draw (7.center) to (4.center);
		\draw (5.center) to (6.center);
		\draw (8) to (12);
		\draw (10) to (11);
		\draw (9) to (13);
	\end{pgfonlayer}
\end{tikzpicture}
\eq{cnot.two}
\begin{tikzpicture}
	\begin{pgfonlayer}{nodelayer}
		\node [style=none] (4) at (7, 1) {};
		\node [style=none] (5) at (7.5, 1) {};
		\node [style=none] (6) at (7, -1) {};
		\node [style=none] (7) at (7.5, -1) {};
	\end{pgfonlayer}
	\begin{pgfonlayer}{edgelayer}
		\draw [in=270, out=90] (7.center) to (4.center);
		\draw [in=90, out=-90] (5.center) to (6.center);
	\end{pgfonlayer}
\end{tikzpicture}
\hspace*{.5cm}
\begin{tikzpicture}
	\begin{pgfonlayer}{nodelayer}
		\node [style=dot] (0) at (8, 0) {};
		\node [style=dot] (1) at (8.5, 0.5) {};
		\node [style=dot] (2) at (8.5, -0.5) {};
		\node [style=oplus] (3) at (8.5, 0) {};
		\node [style=oplus] (4) at (9, 0.5) {};
		\node [style=oplus] (5) at (9, -0.5) {};
		\node [style=none] (6) at (9, -1) {};
		\node [style=none] (7) at (8.5, -1) {};
		\node [style=none] (8) at (8, -1) {};
		\node [style=none] (9) at (8, 1) {};
		\node [style=none] (10) at (8.5, 1) {};
		\node [style=none] (11) at (9, 1) {};
	\end{pgfonlayer}
	\begin{pgfonlayer}{edgelayer}
		\draw (9.center) to (8.center);
		\draw (7.center) to (10.center);
		\draw (11.center) to (6.center);
		\draw (5) to (2);
		\draw (3) to (0);
		\draw (1) to (4);
	\end{pgfonlayer}
\end{tikzpicture}
\eq{cnot.three}
\begin{tikzpicture}
	\begin{pgfonlayer}{nodelayer}
		\node [style=dot] (0) at (8, -0.25) {};
		\node [style=dot] (1) at (8, 0.25) {};
		\node [style=oplus] (3) at (8.5, -0.25) {};
		\node [style=oplus] (4) at (9, 0.25) {};
		\node [style=none] (6) at (9, -1) {};
		\node [style=none] (7) at (8.5, -1) {};
		\node [style=none] (8) at (8, -1) {};
		\node [style=none] (9) at (8, 1) {};
		\node [style=none] (10) at (8.5, 1) {};
		\node [style=none] (11) at (9, 1) {};
	\end{pgfonlayer}
	\begin{pgfonlayer}{edgelayer}
		\draw (9.center) to (8.center);
		\draw (7.center) to (10.center);
		\draw (11.center) to (6.center);
		\draw (3) to (0);
		\draw (1) to (4);
	\end{pgfonlayer}
\end{tikzpicture}
\hspace*{.5cm}
\begin{tikzpicture}
	\begin{pgfonlayer}{nodelayer}
		\node [style=dot] (0) at (2, 1.5) {};
		\node [style=oplus] (1) at (2.5, 1.5) {};
		\node [style=none] (3) at (2, 1) {};
		\node [style=none] (4) at (2.5, 1) {};
		\node [style=none] (5) at (2, 2.5) {};
		\node [style=none] (6) at (2.5, 2.5) {};
		\node [style=none] (7) at (3, 2.5) {};
		\node [style=none] (8) at (3, 1) {};
		\node [style=dot] (9) at (3, 2) {};
		\node [style=oplus] (10) at (2.5, 2) {};
	\end{pgfonlayer}
	\begin{pgfonlayer}{edgelayer}
		\draw (0) to (5.center);
		\draw (6.center) to (1);
		\draw (1) to (0);
		\draw (4.center) to (1);
		\draw (3.center) to (0);
		\draw (10) to (9);
		\draw (8.center) to (7.center);
	\end{pgfonlayer}
\end{tikzpicture}
\eq{cnot.four}
\begin{tikzpicture}
	\begin{pgfonlayer}{nodelayer}
		\node [style=dot] (0) at (2, 2) {};
		\node [style=oplus] (1) at (2.5, 2) {};
		\node [style=none] (3) at (2, 2.5) {};
		\node [style=none] (4) at (2.5, 2.5) {};
		\node [style=none] (5) at (2, 1) {};
		\node [style=none] (6) at (2.5, 1) {};
		\node [style=none] (7) at (3, 1) {};
		\node [style=none] (8) at (3, 2.5) {};
		\node [style=dot] (9) at (3, 1.5) {};
		\node [style=oplus] (10) at (2.5, 1.5) {};
	\end{pgfonlayer}
	\begin{pgfonlayer}{edgelayer}
		\draw (0) to (5.center);
		\draw (6.center) to (1);
		\draw (1) to (0);
		\draw (4.center) to (1);
		\draw (3.center) to (0);
		\draw (10) to (9);
		\draw (8.center) to (7.center);
	\end{pgfonlayer}
\end{tikzpicture}
\hspace*{.5cm}
\begin{tikzpicture}
	\begin{pgfonlayer}{nodelayer}
		\node [style=oplus] (0) at (2, 1.5) {};
		\node [style=dot] (1) at (2.5, 1.5) {};
		\node [style=none] (3) at (2, 1) {};
		\node [style=none] (4) at (2.5, 1) {};
		\node [style=none] (5) at (2, 2.5) {};
		\node [style=none] (6) at (2.5, 2.5) {};
		\node [style=none] (7) at (3, 2.5) {};
		\node [style=none] (8) at (3, 1) {};
		\node [style=oplus] (9) at (3, 2) {};
		\node [style=dot] (10) at (2.5, 2) {};
	\end{pgfonlayer}
	\begin{pgfonlayer}{edgelayer}
		\draw (0) to (5.center);
		\draw (6.center) to (1);
		\draw (1) to (0);
		\draw (4.center) to (1);
		\draw (3.center) to (0);
		\draw (10) to (9);
		\draw (8.center) to (7.center);
	\end{pgfonlayer}
\end{tikzpicture}
\eq{cnot.five}
\begin{tikzpicture}
	\begin{pgfonlayer}{nodelayer}
		\node [style=oplus] (0) at (2, 2) {};
		\node [style=dot] (1) at (2.5, 2) {};
		\node [style=none] (3) at (2, 2.5) {};
		\node [style=none] (4) at (2.5, 2.5) {};
		\node [style=none] (5) at (2, 1) {};
		\node [style=none] (6) at (2.5, 1) {};
		\node [style=none] (7) at (3, 1) {};
		\node [style=none] (8) at (3, 2.5) {};
		\node [style=oplus] (9) at (3, 1.5) {};
		\node [style=dot] (10) at (2.5, 1.5) {};
	\end{pgfonlayer}
	\begin{pgfonlayer}{edgelayer}
		\draw (0) to (5.center);
		\draw (6.center) to (1);
		\draw (1) to (0);
		\draw (4.center) to (1);
		\draw (3.center) to (0);
		\draw (10) to (9);
		\draw (8.center) to (7.center);
	\end{pgfonlayer}
\end{tikzpicture}
\end{align*}
\end{definition}

\begin{lemma} \cite[Thm. 6]{lafont}
$\Iso(\cb_2)$ is a presentation for the prop $(\Iso(\Mat(\F_2),+))$
\end{lemma}

\begin{definition}
Consider the prop $\inj(\cb_2)$ generated by the coproduct of props $\Iso(\cb_2)+\inj$ modulo the equation:
\hspace*{1cm}
$
\begin{tikzpicture}
	\begin{pgfonlayer}{nodelayer}
		\node [style=X] (0) at (0, 0) {};
		\node [style=dot] (1) at (0, 0.5) {};
		\node [style=oplus] (2) at (0.5, 0.5) {};
		\node [style=none] (3) at (0.5, -0.25) {};
		\node [style=none] (4) at (0.5, 1) {};
		\node [style=none] (5) at (0, 1) {};
	\end{pgfonlayer}
	\begin{pgfonlayer}{edgelayer}
		\draw (0) to (1);
		\draw (1) to (5.center);
		\draw (1) to (2);
		\draw (2) to (4.center);
		\draw (3.center) to (2);
	\end{pgfonlayer}
\end{tikzpicture}
\eq{cnot.six}
\begin{tikzpicture}
	\begin{pgfonlayer}{nodelayer}
		\node [style=X] (0) at (0, 0) {};
		\node [style=none] (3) at (0.5, -0.25) {};
		\node [style=none] (4) at (0.5, 0.5) {};
		\node [style=none] (5) at (0, 0.5) {};
	\end{pgfonlayer}
	\begin{pgfonlayer}{edgelayer}
		\draw (0) to (5.center);
		\draw (3.center) to (4.center);
	\end{pgfonlayer}
\end{tikzpicture}
$

\end{definition}

\begin{lemma} \cite[Thm. 7]{lafont}
$\inj(\cb_2)$ is a presentation for the prop $(\inj(\Mat(\F_2)),+)$
\end{lemma}

The white comultiplication can be derived in this fragment:
\hspace*{1cm}$
\left\llbracket
\begin{tikzpicture}
	\begin{pgfonlayer}{nodelayer}
		\node [style=oplus] (0) at (1, 2.75) {};
		\node [style=dot] (1) at (0.5, 2.75) {};
		\node [style=none] (2) at (1, 3.5) {};
		\node [style=none] (3) at (1, 2.25) {};
		\node [style=none] (4) at (0.5, 2) {};
		\node [style=none] (5) at (0.5, 3.5) {};
		\node [style=X] (6) at (1, 2.25) {};
	\end{pgfonlayer}
	\begin{pgfonlayer}{edgelayer}
		\draw (3.center) to (0);
		\draw (0) to (2.center);
		\draw (5.center) to (1);
		\draw (1) to (0);
		\draw (1) to (4.center);
	\end{pgfonlayer}
\end{tikzpicture}
\right\rrbracket
=
\begin{tikzpicture}
	\begin{pgfonlayer}{nodelayer}
		\node [style=X] (0) at (1.25, -1) {};
		\node [style=Z] (1) at (0.75, -1.75) {};
		\node [style=none] (2) at (1.5, -2) {};
		\node [style=none] (3) at (1.25, -0.5) {};
		\node [style=none] (4) at (0.5, -0.5) {};
		\node [style=none] (5) at (0.75, -2.25) {};
		\node [style=X] (6) at (1.5, -2) {};
	\end{pgfonlayer}
	\begin{pgfonlayer}{edgelayer}
		\draw (3.center) to (0);
		\draw [in=90, out=-75] (0) to (2.center);
		\draw (5.center) to (1);
		\draw (1) to (0);
		\draw [in=-90, out=105] (1) to (4.center);
	\end{pgfonlayer}
\end{tikzpicture}
=
\begin{tikzpicture}
	\begin{pgfonlayer}{nodelayer}
		\node [style=Z] (1) at (0.75, -1.75) {};
		\node [style=none] (3) at (1, -1) {};
		\node [style=none] (4) at (0.5, -1) {};
		\node [style=none] (5) at (0.75, -2.25) {};
	\end{pgfonlayer}
	\begin{pgfonlayer}{edgelayer}
		\draw (5.center) to (1);
		\draw [in=-90, out=120] (1) to (4.center);
		\draw [in=-90, out=60] (1) to (3.center);
	\end{pgfonlayer}
\end{tikzpicture}
$


As a matter of notation, given a category $\X$ with finite limits, we refer to the subcategory of $\Span(\X)$ where the left leg is monic as $\Par(\X)$, and the subcategory of spans where all legs are monic by $\Par\Iso(\X)$.  These two categories, respectively, give semantics for partial maps and partially invertible maps in $\X$ (see \cite{cockett} for more details).

\begin{definition}
\label{def:pariso:cb}
Consider the prop $\ParIso(\cb_2)$ generated by the distributive law of props:
$$
\inj(\cb_2)^\op \otimes_{\Iso(\cb_2)} \inj(\cb_2);
\begin{tikzpicture}
	\begin{pgfonlayer}{nodelayer}
		\node [style=X] (0) at (0, 0) {};
		\node [style=X] (1) at (0, 0.75) {};
	\end{pgfonlayer}
	\begin{pgfonlayer}{edgelayer}
		\draw (0) to (1);
	\end{pgfonlayer}
\end{tikzpicture}
\eref{extra}
$$
\end{definition}

To see that this is actually a distributive law, see Remark \ref{rem:pariso:cb}.

\begin{lemma}
\label{lem:parisocb}
$\Par\Iso(\cb_2)$ is a presentation for the prop $(\Par\Iso(\Mat(\F_2),+))$.
\end{lemma}

We can get partial maps by freely adding a counit to the nonunital, noncounital special commutative Frobenius algebra:

\begin{definition}

Let $\Par(\cb_2)$ denote the pushout of the diagram of props:
$$
\Par\Iso(\cb_2)  \leftarrow  \surj^\op \rightarrow   \cm^\op
$$

\end{definition}


\begin{lemma}
\label{lem:parcb}

$\Par(\cb_2)$ is a presentation for the prop $(\Par(\Mat(\F_2),+))$.
\end{lemma}
The proof is contained in \S \ref{proof:lem:parcb}. 

%
%
%

%
%
%
%
%
%
%
%
%
%


\begin{definition}
Let $\Span(\cb_2)$ denote the pushout of the diagram of props:
$$
\Par(\cb_2)^\op \leftarrow  \ParIso(\cb_2) \rightarrow \Par(\cb_2)
$$
\end{definition}

The following lemma holds because of \cite[Lem. 4.3]{zxa}:

\begin{lemma}
\label{lem:spancb}

$\Span(\cb_2)$ is a presentation for the prop $(\Span(\Mat(\F_2)), +)$.
\end{lemma}

The proof is contained in \S \ref{proof:lem:spancb}.
Given a PID $k$, the prop $(\Span(\Mat(k)), +)$ is already known to have a presentation given in terms of ``interacting Hopf algebras" \cite[Definition 3.13]{zanasi}.  This is also the way in which the phase-free fragment of the ZX-calculus would be presented, in terms of two Frobenius algebras  corresponding to the $Z$ and $X$ observables, interacting to form Hopf algebras in addition to a few more equations.

\section{Additive affine models}
\label{sec:two}

\begin{definition}
Consider the prop  $\Aff\cb_2$ given by adjoining the following generator to $\cb_2$
\hfil
$
\begin{tikzpicture}
	\begin{pgfonlayer}{nodelayer}
		\node [style=none] (0) at (-3.75, -0.25) {};
		\node [style=X] (1) at (-3.75, -1) {$\pi$};
	\end{pgfonlayer}
	\begin{pgfonlayer}{edgelayer}
		\draw (0.center) to (1);
	\end{pgfonlayer}
\end{tikzpicture}
$

modulo the equations:\hspace*{2.2cm}
$
\begin{tikzpicture}
	\begin{pgfonlayer}{nodelayer}
		\node [style=X] (0) at (0.75, 0.25) {$\pi$};
		\node [style=Z] (1) at (0.75, 0.75) {};
		\node [style=none] (2) at (0.5, 1.5) {};
		\node [style=none] (3) at (1, 1.5) {};
	\end{pgfonlayer}
	\begin{pgfonlayer}{edgelayer}
		\draw (1) to (0);
		\draw [in=-90, out=60] (1) to (3.center);
		\draw [in=120, out=-90] (2.center) to (1);
	\end{pgfonlayer}
\end{tikzpicture}
\erefop{bi.two}
\begin{tikzpicture}
	\begin{pgfonlayer}{nodelayer}
		\node [style=X] (0) at (0.5, 0.5) {$\pi$};
		\node [style=none] (1) at (0.5, 1.75) {};
		\node [style=none] (2) at (1, 1.75) {};
		\node [style=X] (3) at (1, 0.5) {$\pi$};
	\end{pgfonlayer}
	\begin{pgfonlayer}{edgelayer}
		\draw (0) to (1.center);
		\draw (2.center) to (3);
	\end{pgfonlayer}
\end{tikzpicture}
\hspace*{1cm}
\begin{tikzpicture}
	\begin{pgfonlayer}{nodelayer}
		\node [style=X] (0) at (0, 0) {$\pi$};
		\node [style=Z] (1) at (0, 0.75) {};
	\end{pgfonlayer}
	\begin{pgfonlayer}{edgelayer}
		\draw (1) to (0);
	\end{pgfonlayer}
\end{tikzpicture}
\eref{extra}
$

\end{definition}

\begin{lemma} \cite[\S 4]{lafont}
 $\Aff\cb_2$ is a presentation for the prop $(\Aff\Mat(\F_2),+)$.
\end{lemma}

Note that this assumes that affine matrices are non-empty, as this is a prop.  This will become a problem later, when we wish to pull back affine spaces.

\begin{definition}
Consider the prop $\Iso(\Aff\cb_2)$ generated by the controlled not gate, and the not gate (interpreted as matrices):
\hspace*{.2cm}
$
\left\llbracket
\begin{tikzpicture}
	\begin{pgfonlayer}{nodelayer}
		\node [style=oplus] (5) at (0.5, 2.75) {};
		\node [style=dot] (6) at (0, 2.75) {};
		\node [style=none] (7) at (0.5, 3.5) {};
		\node [style=none] (8) at (0.5, 2) {};
		\node [style=none] (9) at (0, 2) {};
		\node [style=none] (10) at (0, 3.5) {};
	\end{pgfonlayer}
	\begin{pgfonlayer}{edgelayer}
		\draw (8.center) to (5);
		\draw (5) to (7.center);
		\draw (10.center) to (6);
		\draw (6) to (5);
		\draw (6) to (9.center);
	\end{pgfonlayer}
\end{tikzpicture}
\right\rrbracket
=
\begin{tikzpicture}
	\begin{pgfonlayer}{nodelayer}
		\node [style=X] (0) at (-0.25, -1) {};
		\node [style=Z] (1) at (-0.75, -1.75) {};
		\node [style=none] (2) at (0, -2.25) {};
		\node [style=none] (3) at (-0.25, -0.5) {};
		\node [style=none] (4) at (-1, -0.5) {};
		\node [style=none] (5) at (-0.75, -2.25) {};
	\end{pgfonlayer}
	\begin{pgfonlayer}{edgelayer}
		\draw (3.center) to (0);
		\draw [in=90, out=-75] (0) to (2.center);
		\draw (5.center) to (1);
		\draw (1) to (0);
		\draw [in=-90, out=105] (1) to (4.center);
	\end{pgfonlayer}
\end{tikzpicture}
\hspace*{1cm}
\left\llbracket
\begin{tikzpicture}
	\begin{pgfonlayer}{nodelayer}
		\node [style=none] (0) at (0, -0.5) {};
		\node [style=none] (1) at (0, -1.5) {};
		\node [style=oplus] (2) at (0, -1) {};
	\end{pgfonlayer}
	\begin{pgfonlayer}{edgelayer}
		\draw (1.center) to (2);
		\draw (2) to (0.center);
	\end{pgfonlayer}
\end{tikzpicture}
\right\rrbracket
=
\begin{tikzpicture}
	\begin{pgfonlayer}{nodelayer}
		\node [style=none] (0) at (1, 1) {};
		\node [style=none] (1) at (0.75, 2) {};
		\node [style=X] (2) at (0.75, 1.5) {};
		\node [style=X] (3) at (0.5, 1) {$\pi$};
		\node [style=none] (4) at (1, 0.5) {};
	\end{pgfonlayer}
	\begin{pgfonlayer}{edgelayer}
		\draw (1.center) to (2);
		\draw [in=90, out=-45, looseness=0.75] (2) to (0.center);
		\draw [in=90, out=-135, looseness=0.75] (2) to (3);
		\draw (0.center) to (4.center);
	\end{pgfonlayer}
\end{tikzpicture}
$

Modulo the relations of $\Iso(\cb_2)$ as well as the additional relations:
$$
\begin{tikzpicture}
	\begin{pgfonlayer}{nodelayer}
		\node [style=oplus] (0) at (0, 0) {};
		\node [style=oplus] (1) at (0, 0.5) {};
		\node [style=none] (2) at (0, 1) {};
		\node [style=none] (3) at (0, -0.5) {};
	\end{pgfonlayer}
	\begin{pgfonlayer}{edgelayer}
		\draw (2.center) to (3.center);
	\end{pgfonlayer}
\end{tikzpicture}
\eq{cnot.seven}
\begin{tikzpicture}
	\begin{pgfonlayer}{nodelayer}
		\node [style=none] (2) at (0, 1) {};
		\node [style=none] (3) at (0, -0.5) {};
	\end{pgfonlayer}
	\begin{pgfonlayer}{edgelayer}
		\draw (2.center) to (3.center);
	\end{pgfonlayer}
\end{tikzpicture}
\hspace*{.5cm}
\begin{tikzpicture}
	\begin{pgfonlayer}{nodelayer}
		\node [style=dot] (0) at (2, 2) {};
		\node [style=oplus] (1) at (2.5, 2) {};
		\node [style=oplus] (2) at (2, 1.5) {};
		\node [style=none] (3) at (2, 1) {};
		\node [style=none] (4) at (2.5, 1) {};
		\node [style=none] (5) at (2, 2.5) {};
		\node [style=none] (6) at (2.5, 2.5) {};
	\end{pgfonlayer}
	\begin{pgfonlayer}{edgelayer}
		\draw (3.center) to (2);
		\draw (2) to (0);
		\draw (0) to (5.center);
		\draw (6.center) to (1);
		\draw (1) to (0);
		\draw (4.center) to (1);
	\end{pgfonlayer}
\end{tikzpicture}
\eq{cnot.eight}
\begin{tikzpicture}
	\begin{pgfonlayer}{nodelayer}
		\node [style=dot] (0) at (2, 1.5) {};
		\node [style=oplus] (1) at (2.5, 1.5) {};
		\node [style=none] (3) at (2, 1) {};
		\node [style=none] (4) at (2.5, 1) {};
		\node [style=none] (5) at (2, 2.5) {};
		\node [style=none] (6) at (2.5, 2.5) {};
		\node [style=oplus] (7) at (2, 2) {};
		\node [style=oplus] (8) at (2.5, 2) {};
	\end{pgfonlayer}
	\begin{pgfonlayer}{edgelayer}
		\draw (0) to (5.center);
		\draw (6.center) to (1);
		\draw (1) to (0);
		\draw (4.center) to (1);
		\draw (3.center) to (0);
	\end{pgfonlayer}
\end{tikzpicture}
\hspace*{.5cm}
\begin{tikzpicture}
	\begin{pgfonlayer}{nodelayer}
		\node [style=dot] (0) at (2, 2) {};
		\node [style=oplus] (1) at (2.5, 2) {};
		\node [style=none] (3) at (2, 1) {};
		\node [style=none] (4) at (2.5, 1) {};
		\node [style=none] (5) at (2, 2.5) {};
		\node [style=none] (6) at (2.5, 2.5) {};
		\node [style=oplus] (8) at (2.5, 1.5) {};
	\end{pgfonlayer}
	\begin{pgfonlayer}{edgelayer}
		\draw (0) to (5.center);
		\draw (6.center) to (1);
		\draw (1) to (0);
		\draw (4.center) to (1);
		\draw (3.center) to (0);
	\end{pgfonlayer}
\end{tikzpicture}
\eq{cnot.nine}
\begin{tikzpicture}
	\begin{pgfonlayer}{nodelayer}
		\node [style=dot] (0) at (2, 1.5) {};
		\node [style=oplus] (1) at (2.5, 1.5) {};
		\node [style=none] (3) at (2, 1) {};
		\node [style=none] (4) at (2.5, 1) {};
		\node [style=none] (5) at (2, 2.5) {};
		\node [style=none] (6) at (2.5, 2.5) {};
		\node [style=oplus] (8) at (2.5, 2) {};
	\end{pgfonlayer}
	\begin{pgfonlayer}{edgelayer}
		\draw (0) to (5.center);
		\draw (6.center) to (1);
		\draw (1) to (0);
		\draw (4.center) to (1);
		\draw (3.center) to (0);
	\end{pgfonlayer}
\end{tikzpicture}
$$

\end{definition}

\begin{lemma}\cite[Thm. 11]{lafont}
$\Iso(\Aff\cb_2)$ is a presentation for the prop $(\Iso(\Aff\Mat(\F_2),+))$.
\end{lemma}

\begin{definition}
Let $\inj(\Aff\cb_2)$ denote the pushout of the diagram of props:
$$
 \inj(\cb_2) \leftarrow  \Iso(\cb_2)\rightarrow  \Iso(\Aff\cb_2)
$$
\end{definition}

\begin{lemma}
\label{lem:injaffcb}
$\inj(\Aff\cb_2)$ is a presentation for the prop $(\inj(\Aff\Mat(\F_2)),+)$.
\end{lemma}

The proof is contained in \S \ref{lem:injaffcb}.

To define partial isomorphisms, we must add a point to the constituent props of the desired distributive law, because the empty set can arise as a subobject by pullback (where the empty set is not properly an object in the prop).

%
%
%

\begin{definition}
\label{def:isoaffcbzero}
Let $\Iso(\Aff\cb_2)^{+1}$ denote the prop obtained by adjoining the following generator to $\Iso(\Aff\cb_2)$ 
$
\begin{tikzpicture}
	\begin{pgfonlayer}{nodelayer}
		\node [style=X] (0) at (0, 0) {$\pi$};
	\end{pgfonlayer}
\end{tikzpicture}
$
modulo the equations:
$$
\begin{tikzpicture}
	\begin{pgfonlayer}{nodelayer}
		\node [style=X] (0) at (0, 0) {$\pi$};
		\node [style=X] (3) at (0.5, 0) {$\pi$};
	\end{pgfonlayer}
\end{tikzpicture}
\eq{zero.one}
\begin{tikzpicture}
	\begin{pgfonlayer}{nodelayer}
		\node [style=X] (0) at (0, 0) {$\pi$};
	\end{pgfonlayer}
\end{tikzpicture},
\hspace*{.5cm}
\begin{tikzpicture}
	\begin{pgfonlayer}{nodelayer}
		\node [style=X] (0) at (0, 1) {$\pi$};
		\node [style=none] (1) at (0.5, 0.5) {};
		\node [style=none] (2) at (0.5, 1.5) {};
		\node [style=none] (3) at (1, 1.5) {};
		\node [style=none] (4) at (1, 0.5) {};
		\node [style=dot] (5) at (0.5, 1) {};
		\node [style=oplus] (6) at (1, 1) {};
	\end{pgfonlayer}
	\begin{pgfonlayer}{edgelayer}
		\draw [in=90, out=-90] (2.center) to (1.center);
		\draw [in=-90, out=90] (4.center) to (3.center);
		\draw (6) to (5);
	\end{pgfonlayer}
\end{tikzpicture}
\eq{zero.two}
\begin{tikzpicture}
	\begin{pgfonlayer}{nodelayer}
		\node [style=X] (0) at (0, 1) {$\pi$};
		\node [style=none] (1) at (0.5, 0.5) {};
		\node [style=none] (2) at (0.5, 1.5) {};
		\node [style=none] (3) at (1, 1.5) {};
		\node [style=none] (4) at (1, 0.5) {};
	\end{pgfonlayer}
	\begin{pgfonlayer}{edgelayer}
		\draw [in=90, out=-90] (2.center) to (1.center);
		\draw [in=-90, out=90] (4.center) to (3.center);
	\end{pgfonlayer}
\end{tikzpicture},
\hspace*{.5cm}
\begin{tikzpicture}
	\begin{pgfonlayer}{nodelayer}
		\node [style=X] (0) at (0, 1) {$\pi$};
		\node [style=none] (1) at (0.5, 0.5) {};
		\node [style=none] (2) at (1, 1.5) {};
		\node [style=none] (3) at (0.5, 1.5) {};
		\node [style=none] (4) at (1, 0.5) {};
	\end{pgfonlayer}
	\begin{pgfonlayer}{edgelayer}
		\draw [in=90, out=-90] (2.center) to (1.center);
		\draw [in=-90, out=90] (4.center) to (3.center);
	\end{pgfonlayer}
\end{tikzpicture}
\eq{zero.three}
\begin{tikzpicture}
	\begin{pgfonlayer}{nodelayer}
		\node [style=X] (0) at (0, 1) {$\pi$};
		\node [style=none] (1) at (0.5, 0.5) {};
		\node [style=none] (2) at (0.5, 1.5) {};
		\node [style=none] (3) at (1, 1.5) {};
		\node [style=none] (4) at (1, 0.5) {};
	\end{pgfonlayer}
	\begin{pgfonlayer}{edgelayer}
		\draw [in=90, out=-90] (2.center) to (1.center);
		\draw [in=-90, out=90] (4.center) to (3.center);
	\end{pgfonlayer}
\end{tikzpicture},
\hspace*{.5cm}
\begin{tikzpicture}
	\begin{pgfonlayer}{nodelayer}
		\node [style=X] (0) at (0, 1) {$\pi$};
		\node [style=none] (1) at (0.5, 0.5) {};
		\node [style=none] (2) at (0.5, 1.5) {};
		\node [style=oplus] (3) at (0.5, 1) {};
	\end{pgfonlayer}
	\begin{pgfonlayer}{edgelayer}
		\draw (2.center) to (1.center);
	\end{pgfonlayer}
\end{tikzpicture}
\eq{zero.four}
\begin{tikzpicture}
	\begin{pgfonlayer}{nodelayer}
		\node [style=X] (0) at (0, 1) {$\pi$};
		\node [style=none] (1) at (0.5, 0.5) {};
		\node [style=none] (2) at (0.5, 1.5) {};
	\end{pgfonlayer}
	\begin{pgfonlayer}{edgelayer}
		\draw (2.center) to (1.center);
	\end{pgfonlayer}
\end{tikzpicture}
$$
\end{definition}

\begin{lemma}
$\Iso(\Aff\cb_2)^{+1}$ is a presentation for the subcategory of $(\Span(\Aff\Fin\Vect(\F_2)), +)$ generated by spans $\F_2^n = \F_2^n \xrightarrow[\cong]{f} \F_2^n$ and $\F_2^n \xleftarrowtail {?} \emptyset \xrightarrowtail{?}  \F_2^n$, for all $n \in \N$ and isomorphisms $f$. 
\end{lemma}

%
%
%

\begin{proof}
Identify this new generator with the span $\F_2^0 \leftarrow \emptyset \rightarrow \F_2^0$.  If there is a factor of 
$
\begin{tikzpicture}
	\begin{pgfonlayer}{nodelayer}
		\node [style=X] (0) at (0, 0) {$\pi$};
	\end{pgfonlayer}
\end{tikzpicture}
$,   repeatedly apply these identities from left to right until the diagram corresponding to the identity tensored by $
\begin{tikzpicture}
	\begin{pgfonlayer}{nodelayer}
		\node [style=X] (0) at (0, 0) {$\pi$};
	\end{pgfonlayer}
\end{tikzpicture}
$ is obtained, which is as a normal form.
\end{proof}

\begin{definition}
Let $\inj(\Aff\cb_2)^{+1}$ denote the pushout of the diagram of props:

$$
\inj(\Aff\cb_2) \leftarrow \Iso(\Aff\cb_2) \rightarrow \Iso(\Aff\cb_2)^{+1}
$$

\end{definition}

\begin{lemma}
$\inj(\Aff\cb_2)^{+1}$ is a presentation for the subcategory of $(\Span(\Aff\Fin\Vect(\F_2)), +)$ generated by spans $\F_2^n = \F_2^n \xrightarrowtail{e} \F_2^m$ and $\F_2^n \xleftarrowtail{?} \emptyset \xrightarrowtail{?}  \F_2^n$, for all $n,m \in \N$ and monics $e$. 
\end{lemma}
%

The proof of this lemma is essentially the same for $\Iso(\Aff\cb_2)^{+1}$, although diagrams with a factor of
$
\begin{tikzpicture}
	\begin{pgfonlayer}{nodelayer}
		\node [style=X] (0) at (0, 0) {$\pi$};
	\end{pgfonlayer}
\end{tikzpicture}
$ are reduced to the following normal form:
\hspace*{3cm}
$
\begin{tikzpicture}
	\begin{pgfonlayer}{nodelayer}
		\node [style=X] (0) at (0, 1.25) {$\pi$};
		\node [style=none] (1) at (0.5, 0.5) {};
		\node [style=none] (2) at (0.5, 1.75) {};
		\node [style=none] (3) at (1, 0.5) {};
		\node [style=none] (4) at (1, 1.75) {};
		\node [style=X] (5) at (1.5, 0.75) {};
		\node [style=X] (6) at (2, 0.75) {};
		\node [style=none] (7) at (1.5, 1.75) {};
		\node [style=none] (8) at (2, 1.75) {};
		\node [style=none] (9) at (0.75, 1.5) {$n$};
		\node [style=none] (10) at (1.75, 1.5) {$m$};
		\node [style=none] (11) at (1.77, 1.25) {$\cdots$};
		\node [style=none] (12) at (0.77, 1.25) {$\cdots$};
	\end{pgfonlayer}
	\begin{pgfonlayer}{edgelayer}
		\draw (2.center) to (1.center);
		\draw (4.center) to (3.center);
		\draw (5) to (7.center);
		\draw (8.center) to (6);
	\end{pgfonlayer}
\end{tikzpicture}
$

Unlike in the linear case, now we must consider a distributive law over a prop which is not a groupoid: we add a single idempotent corresponding to the empty set to the isomorphisms.  To satisfy the requirement that this prop is a sub-prop of the left and right components of the  distributive law, we also add this idempotent to the injections and the co-injections:

\begin{definition}
\label{def:parisoaffcb}
Consider the prop $\pr\iso\Aff\cb_2$ generated by the distributive law of props:

$$
 (\inj(\Aff\cb_2)^{+1})^\op \otimes_{\Iso(\Aff\cb_2)^{+1}}  \inj(\Aff\cb_2)^{+1}
$$
Given by the equations of $\pr\iso\Aff\cb_2$ as well as:
\!
$
\begin{tikzpicture}
	\begin{pgfonlayer}{nodelayer}
		\node [style=X] (0) at (0.5, 0.75) {$\pi$};
		\node [style=X] (1) at (0.5, 0) {};
	\end{pgfonlayer}
	\begin{pgfonlayer}{edgelayer}
		\draw (0) to (1);
	\end{pgfonlayer}
\end{tikzpicture}
=
\begin{tikzpicture}
	\begin{pgfonlayer}{nodelayer}
		\node [style=X] (0) at (0, 0) {$\pi$};
		\node [style=X] (1) at (0, 0.75) {};
	\end{pgfonlayer}
	\begin{pgfonlayer}{edgelayer}
		\draw (0) to (1);
	\end{pgfonlayer}
\end{tikzpicture}
\eq{zero.five}
\begin{tikzpicture}
	\begin{pgfonlayer}{nodelayer}
		\node [style=X] (0) at (0, 0) {$\pi$};
	\end{pgfonlayer}
	\begin{pgfonlayer}{edgelayer}
	\end{pgfonlayer}
\end{tikzpicture}
$

\end{definition}

To see that this is actually a distributive law, see Remark \ref{rem:parisoaffcb}.

\begin{lemma}
\label{lem:parisoaffcb}
$\ParIso(\Aff\cb_2)$ is a presentation for the full subcategory $\Par\Iso(\Aff\Fin\Vect(\F_2))^*$ of $\Par\Iso(\Aff\Fin\Vect(\F_2))$ where the objects are nonempty affine vector spaces.
\end{lemma}

\begin{proof}
The obvious functor $\ParIso(\Aff\cb_2)\to \Par\Iso(\Aff\Fin\Vect(\F_2))^*$ is clearly full,  as well as an isomorphism on objects.
It remains to show it is faihful.  It is faithful on maps which are taken to spans with nonempty apex by the same argument as Lemma \ref{lem:parisocb}. For empty case, there is exactly one diagram of each type with a factor of $0$; and similarly, there is exactly one span with an empty apex.
\end{proof}

By \cite{cnot}  in this the identities of Definition \ref{def:isoaffcbzero}
 can be replaced by the following identity, while maintaining completeness:
\hspace*{.4cm}
$
\begin{tikzpicture}
	\begin{pgfonlayer}{nodelayer}
		\node [style=X] (0) at (0, 5) {$\pi$};
		\node [style=none] (1) at (0.5, 5.75) {};
		\node [style=none] (2) at (0.5, 4.25) {};
	\end{pgfonlayer}
	\begin{pgfonlayer}{edgelayer}
		\draw (2.center) to (1.center);
	\end{pgfonlayer}
\end{tikzpicture}
\eq{zero.six}
\begin{tikzpicture}
	\begin{pgfonlayer}{nodelayer}
		\node [style=X] (0) at (0, 5) {$\pi$};
		\node [style=none] (1) at (0.5, 5.75) {};
		\node [style=none] (2) at (0.5, 4.25) {};
		\node [style=X] (3) at (0.5, 5.25) {$\pi$};
		\node [style=X] (4) at (0.5, 4.75) {$\pi$};
	\end{pgfonlayer}
	\begin{pgfonlayer}{edgelayer}
		\draw (3) to (1.center);
		\draw (4) to (2.center);
	\end{pgfonlayer}
\end{tikzpicture}
$


%
%

\begin{definition}

Let $\pr\Aff\cb_2$ denote the pushout of the diagram of props:
$$
\pr\iso\Aff\cb_2 \leftarrow \surj^\op \rightarrow \cm^\op
$$

\end{definition}

\begin{lemma}
\label{lem:paraffcb}

$\pr\Aff\cb_2$ is a presentation for the prop $(\Par(\Aff\Fin\Vect(\F_2))^*,+)$.
\end{lemma}
The proof is contained in \S \ref{proof:lem:paraffcb}. 

\begin{definition}
Let $\sp\Aff\cb_2$ denote the pushout of the diagram of props:
$$
 \pr\Aff\cb_2^\op \leftarrow \pr\iso\Aff\cb_2 \rightarrow \pr\Aff\cb_2
$$

\end{definition}

\begin{lemma}
\label{lem:spanaffcb}
$\sp\Aff\cb_2$ is a presentation for the prop $(\Span(\Aff\Fin\Vect(\F_2))^*,+)$.
\end{lemma}
The proof is contained in \S \ref{proof:lem:spanaffcb}.
This is almost equivalent to the presentation given in \cite{affine} which gives a presentation for the full subcategory of relations of finite dimensional affine vector spaces where the objects are given by the nonempty vector spaces, and is much more in the spirit of the ZX-calculus.

\section{The and gate}
\label{sec:three}

Recall that unlike when the tensor product is the coproduct; when the tensor product is induced by the multiplication, to obtain a prop, one must consider the subcategory generated by tensoring a fixed object with itsef.
%

\begin{definition}
Let $L_{\F_2^\times}$ be the prop generated by quotienting $\cb$ by the equation:
\!
$
\begin{tikzpicture}
	\begin{pgfonlayer}{nodelayer}
		\node [style=none] (0) at (-7, 1) {};
		\node [style=none] (1) at (-7, 0.5) {};
		\node [style=Z] (2) at (-7, -0.25) {};
		\node [style=none] (3) at (-7, -0.75) {};
		\node [style=andin] (4) at (-7, 0.5) {};
	\end{pgfonlayer}
	\begin{pgfonlayer}{edgelayer}
		\draw (3.center) to (2.center);
		\draw [in=-60, out=60, looseness=1.25] (2.center) to (1);
		\draw [in=120, out=-120, looseness=1.25] (1) to (2.center);
		\draw (1) to (0.center);
	\end{pgfonlayer}
\end{tikzpicture}
\eq{antispecial}
\begin{tikzpicture}
	\begin{pgfonlayer}{nodelayer}
		\node [style=none] (0) at (-7, 1) {};
		\node [style=none] (1) at (-7, -0.75) {};
	\end{pgfonlayer}
	\begin{pgfonlayer}{edgelayer}
		\draw (1.center) to (0.center);
	\end{pgfonlayer}
\end{tikzpicture}
$

Where the components of the  monoid are relabled as follows:
\hspace*{.5cm}
$
\left(
\begin{tikzpicture}
	\begin{pgfonlayer}{nodelayer}
		\node [style=none] (0) at (-3.75, 0.5) {};
		\node [style=none] (1) at (-3.75, -0.25) {};
		\node [style=andin] (2) at (-3.75, -0.25) {};
		\node [style=none] (3) at (-4, -1) {};
		\node [style=none] (4) at (-3.5, -1) {};
	\end{pgfonlayer}
	\begin{pgfonlayer}{edgelayer}
		\draw (0.center) to (1.center);
		\draw [in=-60, out=90, looseness=1.00] (4.center) to (1.center);
		\draw [in=90, out=-120, looseness=1.00] (1.center) to (3.center);
	\end{pgfonlayer}
\end{tikzpicture},
\begin{tikzpicture}
	\begin{pgfonlayer}{nodelayer}
		\node [style=none] (0) at (-3.75, -0.25) {};
		\node [style=X] (1) at (-3.75, -1) {$\pi$};
	\end{pgfonlayer}
	\begin{pgfonlayer}{edgelayer}
		\draw (0.center) to (1);
	\end{pgfonlayer}
\end{tikzpicture}
\right)
$

\end{definition}

\begin{lemma}
$L_{\F_2}^\times$ is a presentation for the Lawvere theory for the group of units of the field $\F_2$.
\end{lemma}

\begin{definition}
Consider the prop $\f_2$, generated by the distributive law:
$$
L_{\F_2^\times} \otimes_{\cm^\op} \cb_2;
\begin{tikzpicture}
	\begin{pgfonlayer}{nodelayer}
		\node [style=andin] (4) at (1.25, 0.5) {};
		\node [style=X] (5) at (0.75, -0.5) {};
		\node [style=none] (6) at (0.5, -1) {};
		\node [style=none] (7) at (1, -1) {};
		\node [style=none] (8) at (1.75, -1) {};
		\node [style=none] (9) at (1.25, 0.5) {};
		\node [style=none] (10) at (1.25, 1.5) {};
	\end{pgfonlayer}
	\begin{pgfonlayer}{edgelayer}
		\draw [in=-30, out=90] (8.center) to (9.center);
		\draw [in=90, out=-150] (9.center) to (5);
		\draw [in=90, out=-45] (5) to (7.center);
		\draw [in=-135, out=90] (6.center) to (5);
		\draw (9.center) to (10.center);
	\end{pgfonlayer}
\end{tikzpicture}
\eq{ring.mul}
\begin{tikzpicture}
	\begin{pgfonlayer}{nodelayer}
		\node [style=none] (0) at (1, 0) {};
		\node [style=none] (1) at (0.5, -1.25) {};
		\node [style=none] (2) at (1.75, -0.75) {};
		\node [style=none] (3) at (1.33, 0.75) {};
		\node [style=andin] (4) at (1, 0) {};
		\node [style=none] (5) at (1.75, 0) {};
		\node [style=none] (6) at (1, -1.25) {};
		\node [style=none] (7) at (1.75, -0.75) {};
		\node [style=none] (8) at (1.33, 0.75) {};
		\node [style=andin] (9) at (1.75, 0) {};
		\node [style=X] (10) at (1.33, 0.75) {};
		\node [style=none] (11) at (1.33, 1.25) {};
		\node [style=none] (12) at (1.75, -1.25) {};
		\node [style=Z] (13) at (1.75, -0.75) {};
	\end{pgfonlayer}
	\begin{pgfonlayer}{edgelayer}
		\draw [in=-135, out=90] (0.center) to (3.center);
		\draw [in=165, out=-30, looseness=1.25] (0.center) to (2.center);
		\draw [in=-45, out=90] (5.center) to (8.center);
		\draw [in=45, out=-45, looseness=1.25] (5.center) to (7.center);
		\draw (10) to (11.center);
		\draw [in=90, out=-150] (4) to (1.center);
		\draw [in=-150, out=90] (6.center) to (9);
		\draw (12.center) to (13);
	\end{pgfonlayer}
\end{tikzpicture},
\hspace*{.5cm}
\begin{tikzpicture}
	\begin{pgfonlayer}{nodelayer}
		\node [style=none] (0) at (2, 0) {};
		\node [style=none] (1) at (1.75, -0.75) {};
		\node [style=none] (2) at (2.25, -0.75) {};
		\node [style=none] (3) at (2, 0.5) {};
		\node [style=none] (4) at (2.25, -1) {};
		\node [style=X] (5) at (1.75, -0.75) {};
		\node [style=andin] (6) at (2, 0) {};
	\end{pgfonlayer}
	\begin{pgfonlayer}{edgelayer}
		\draw (0.center) to (3.center);
		\draw [in=90, out=-45] (0.center) to (2.center);
		\draw (4.center) to (2.center);
		\draw [in=-135, out=90] (1.center) to (0.center);
	\end{pgfonlayer}
\end{tikzpicture}
\eq{ring.unit}
\begin{tikzpicture}
	\begin{pgfonlayer}{nodelayer}
		\node [style=none] (12) at (2, 0.5) {};
		\node [style=none] (14) at (2, -1) {};
		\node [style=X] (15) at (2, 0) {};
		\node [style=Z] (16) at (2, -0.5) {};
	\end{pgfonlayer}
	\begin{pgfonlayer}{edgelayer}
		\draw (15) to (12.center);
		\draw (16) to (14.center);
	\end{pgfonlayer}
\end{tikzpicture}
$$

\end{definition}

\begin{lemma} \cite[Thm. 10]{lafont}
$\f_2$ is a presentation for the prop $(\FSets_2,\times)$.
\end{lemma}

Therefore in some sense, we are justified in thinking of this prop $(\FSets_2,\times)$ as a sort of categorification of boolean polynomials.

%
%
%
%
%

To find larger fragments, it will be useful to first identify the isomorphisms and the monics of $\f_2$.

\begin{definition}
Given a map $f$ in  $\f_2$, the {\bf oracle} for $f$, ${\mathcal O}_f$ is defined as follows:
$$
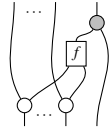
\begin{tikzpicture}
	\begin{pgfonlayer}{nodelayer}
		\node [style=Z] (0) at (0.75, 0.25) {};
		\node [style=X] (1) at (1.5, 2.25) {};
		\node [style=map] (2) at (1, 1.5) {$f$};
		\node [style=none] (3) at (0.5, 2.75) {};
		\node [style=none] (4) at (1.5, 2.75) {};
		\node [style=none] (5) at (1.5, -0.25) {};
		\node [style=none] (6) at (0.75, -0.25) {};
		\node [style=Z] (7) at (-0.25, 0.25) {};
		\node [style=none] (8) at (-0.5, 2.75) {};
		\node [style=none] (9) at (-0.25, -0.25) {};
		\node [style=none] (10) at (0.25, 0) {$\cdots$};
		\node [style=none] (11) at (0, 2.5) {$\cdots$};
	\end{pgfonlayer}
	\begin{pgfonlayer}{edgelayer}
		\draw (6.center) to (0);
		\draw [in=-60, out=60] (0) to (2);
		\draw [in=-120, out=90] (2) to (1);
		\draw (1) to (4.center);
		\draw [in=90, out=-60] (1) to (5.center);
		\draw [in=-90, out=120] (0) to (3.center);
		\draw (9.center) to (7);
		\draw [in=-90, out=120] (7) to (8.center);
		\draw [in=45, out=-120] (2) to (7);
	\end{pgfonlayer}
\end{tikzpicture}
$$

\end{definition}

\begin{lemma}
The oracles in $f_2$ are generated by the generalized controlled-not gates:
$$
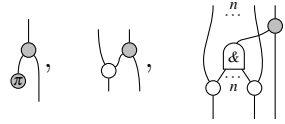
\begin{tikzpicture}
	\begin{pgfonlayer}{nodelayer}
		\node [style=none] (0) at (1, -0.75) {};
		\node [style=X] (1) at (0.5, -0.75) {$\pi$};
		\node [style=none] (2) at (0.75, 0.75) {};
		\node [style=X] (3) at (0.75, 0) {};
		\node [style=none] (4) at (1, -1.25) {};
	\end{pgfonlayer}
	\begin{pgfonlayer}{edgelayer}
		\draw [in=-45, out=90, looseness=0.75] (0.center) to (3);
		\draw [in=90, out=-135, looseness=0.75] (3) to (1);
		\draw (3) to (2.center);
		\draw (4.center) to (0.center);
	\end{pgfonlayer}
\end{tikzpicture},
\hspace*{.5cm}
\begin{tikzpicture}[xscale=-1]
	\begin{pgfonlayer}{nodelayer}
		\node [style=X] (0) at (0.75, 0) {};
		\node [style=Z] (1) at (1.25, -0.5) {};
		\node [style=none] (2) at (0.5, -1) {};
		\node [style=none] (3) at (1.25, -1) {};
		\node [style=none] (4) at (1.5, 0.5) {};
		\node [style=none] (5) at (0.75, 0.5) {};
	\end{pgfonlayer}
	\begin{pgfonlayer}{edgelayer}
		\draw (5.center) to (0);
		\draw [in=150, out=-30] (0) to (1);
		\draw [in=-90, out=60, looseness=0.75] (1) to (4.center);
		\draw (1) to (3.center);
		\draw [in=90, out=-120, looseness=0.75] (0) to (2.center);
	\end{pgfonlayer}
\end{tikzpicture},
\hspace*{.5cm}
\begin{tikzpicture}
	\begin{pgfonlayer}{nodelayer}
		\node [style=Z] (0) at (-10.25, 0.25) {};
		\node [style=Z] (1) at (-11.25, 0.25) {};
		\node [style=none] (2) at (-10.75, 1) {};
		\node [style=X] (3) at (-9.75, 1.75) {};
		\node [style=none] (4) at (-11.25, -0.5) {};
		\node [style=none] (5) at (-10.25, -0.5) {};
		\node [style=none] (6) at (-9.75, -0.5) {};
		\node [style=none] (7) at (-9.75, 2.25) {};
		\node [style=none] (8) at (-10.25, 2.25) {};
		\node [style=none] (9) at (-11.25, 2.25) {};
		\node [style=andin] (10) at (-10.75, 1) {};
		\node [style=none] (11) at (-10.75, 2.25) {$n$};
		\node [style=none] (12) at (-10.75, 0.25) {$n$};
		\node [style=none] (13) at (-10.75, 2) {$\cdots$};
		\node [style=none] (14) at (-10.75, 0.5) {$\cdots$};
	\end{pgfonlayer}
	\begin{pgfonlayer}{edgelayer}
		\draw (4.center) to (1);
		\draw (1) to (2.center);
		\draw (2.center) to (0);
		\draw (0) to (5.center);
		\draw (6.center) to (3);
		\draw [in=90, out=-146, looseness=1.50] (3) to (2.center);
		\draw [in=-90, out=120, looseness=1.00] (1) to (9.center);
		\draw [in=-90, out=60, looseness=0.75] (0) to (8.center);
		\draw (3) to (7.center);
	\end{pgfonlayer}
\end{tikzpicture}
$$



%
%

\end{lemma}

\begin{proof}
Any  Boolean function of $n$ arguments can be represented by a polynomial in\\
 $\F_2[x_1,\ldots, x_n]/\langle x_1^2-x_1,\ldots x_1^2-x_1\rangle$.  Every polynomial in this quotient ring has a unique normal form given by sums of products (which is not true for arbitrary finite fields).  Each product corresponds to a generalized controlled-not gate, and the sum corresponds to composing these generalized controlled-not gates in sequence.
\end{proof}

In the quantum circuit notation the generalized controlled-not gates are drawn as follows (the first being the not gate, and the second being the controlled-not gate):
$$
\begin{tikzpicture}
	\begin{pgfonlayer}{nodelayer}
		\node [style=oplus] (0) at (0, 1.5) {};
		\node [style=none] (1) at (0, 2) {};
		\node [style=none] (2) at (0, 1) {};
	\end{pgfonlayer}
	\begin{pgfonlayer}{edgelayer}
		\draw (0) to (1.center);
		\draw (0) to (2.center);
	\end{pgfonlayer}
\end{tikzpicture}
,
\hspace*{.5cm}
\begin{tikzpicture}
	\begin{pgfonlayer}{nodelayer}
		\node [style=oplus] (0) at (0, 1.5) {};
		\node [style=none] (1) at (0, 2) {};
		\node [style=none] (2) at (0, 1) {};
		\node [style=none] (4) at (-0.5, 2) {};
		\node [style=none] (5) at (-0.5, 1) {};
		\node [style=dot] (6) at (-0.5, 1.5) {};
	\end{pgfonlayer}
	\begin{pgfonlayer}{edgelayer}
		\draw (0) to (1.center);
		\draw (0) to (2.center);
		\draw (0) to (6);
		\draw (6) to (4.center);
		\draw (6) to (5.center);
	\end{pgfonlayer}
\end{tikzpicture}
,
\hspace*{.5cm}
\begin{tikzpicture}
	\begin{pgfonlayer}{nodelayer}
		\node [style=oplus] (0) at (-0.25, 1.5) {};
		\node [style=none] (1) at (-0.25, 2) {};
		\node [style=none] (2) at (-0.25, 1) {};
		\node [style=none] (4) at (-0.75, 2) {};
		\node [style=none] (5) at (-0.75, 1) {};
		\node [style=dot] (6) at (-0.75, 1.5) {};
		\node [style=none] (7) at (-1.75, 2) {};
		\node [style=none] (8) at (-1.75, 1) {};
		\node [style=dot] (9) at (-1.75, 1.5) {};
		\node [style=none] (10) at (-1, 1.5) {};
		\node [style=none] (11) at (-1.5, 1.5) {};
		\node [style=none] (12) at (-1.25, 1.5) {$\cdots$};
		\node [style=none] (13) at (-1.25, 1.75) {$n$};
	\end{pgfonlayer}
	\begin{pgfonlayer}{edgelayer}
		\draw (0) to (1.center);
		\draw (0) to (2.center);
		\draw (0) to (6);
		\draw (6) to (4.center);
		\draw (6) to (5.center);
		\draw (9) to (7.center);
		\draw (9) to (8.center);
		\draw (11.center) to (9);
		\draw (10.center) to (6);
	\end{pgfonlayer}
\end{tikzpicture}
$$


\begin{lemma}\cite[Thm. 5.1]{toffolireversible}
The prop generated by the oracles in $\f_2$ generate $\Iso(\f_2)$.
\end{lemma}


Denote a generalized controlled not gate controlled by wires indexed by $X$, operating on $x$ by $\lbparen X,x\rbparen$

%
%

Iwama et al \cite{iwama} originally gave a complete set of identities for circuits generated by generalized controlled not gates where the value of all-but-one output wires are fixed.  It is worth mentioning that Shende et al. later used the commutator to generalize some of these identities \cite[Cor. 26]{shende}.  We conjecture that a very similar set of identities is complete for Boolean isomorphisms: 


\begin{conjecture}

Let  $\Iso(\FSets_2)$ denote the prop generated by all generalized controlled-not gates modulo the following identities:

\begin{itemize}
\item $\lbparen X,x\rbparen \lbparen X,x \rbparen= 1$ 

\item
If  $x \notin Y $ and $ y \notin X$ then $\lbparen X,x\rbparen \lbparen Y,y\rbparen =\lbparen Y,y\rbparen \lbparen X,x\rbparen $.

%

\item
If $x \notin Y$, then $\lbparen X,x\rbparen \lbparen \{x\} \sqcup Y, y\rbparen = \lbparen X\cup Y,y\rbparen  \lbparen \{x\} \sqcup Y, y\rbparen  \lbparen X,x\rbparen $.

\item
If $x \notin Y$, then $ \lbparen \{x\} \sqcup Y, y\rbparen \lbparen X,x\rbparen = \lbparen X,x\rbparen   \lbparen \{x\} \sqcup Y, y\rbparen  \lbparen X\cup Y,y\rbparen $.


\item
If $x \in Y$ and $y \in X$, then
$
\lbparen  X,x \rbparen  \lbparen  Y,y \rbparen   \lbparen  X,x \rbparen 
=
\lbparen  Y,y \rbparen  \lbparen  X,x \rbparen   \lbparen  Y,y \rbparen 
$.

\end{itemize}

\end{conjecture}

Note that the symmetry is derived in this fragment by composing 3 controlled not gates, as in Definition \label{def:isoaff}.  The axioms of a prop are derived, so we are justified in calling $\Iso(\f_2)$ a prop.

Although we aren't sure if these identities are complete, it doesn't matter in the end.  With each generator we add, we add new enough identities to give a complete presentation, given that there is a complete presentation for $\Iso(\f_2)$.  However, eventually once we add enough generators and identities, we get a finite, complete presentation.

\begin{definition}
Let $\inj(\f_2)$ be the prop given by adjoining the black unit to $\Iso(\f_2)$ modulo:

$$
\begin{tikzpicture}
	\begin{pgfonlayer}{nodelayer}
		\node [style=oplus] (0) at (2, 1.5) {};
		\node [style=none] (1) at (2, 2.25) {};
		\node [style=none] (2) at (2, 0.75) {};
		\node [style=none] (3) at (1.5, 2.25) {};
		\node [style=none] (4) at (1.5, 0.75) {};
		\node [style=dot] (5) at (1.5, 1.5) {};
		\node [style=none] (6) at (0.5, 2.25) {};
		\node [style=dot] (7) at (0.5, 1.5) {};
		\node [style=none] (8) at (1.25, 1.5) {};
		\node [style=none] (9) at (0.75, 1.5) {};
		\node [style=none] (10) at (1, 1.5) {$\cdots$};
		\node [style=none] (11) at (1, 1.75) {$n$};
		\node [style=none] (13) at (0, 2.25) {};
		\node [style=dot] (14) at (0, 1.5) {};
		\node [style=X] (15) at (0, 1) {};
		\node [style=none] (16) at (0.5, 0.75) {};
	\end{pgfonlayer}
	\begin{pgfonlayer}{edgelayer}
		\draw (0) to (1.center);
		\draw (0) to (2.center);
		\draw (0) to (5);
		\draw (5) to (3.center);
		\draw (5) to (4.center);
		\draw (7) to (6.center);
		\draw (9.center) to (7);
		\draw (8.center) to (5);
		\draw (14) to (13.center);
		\draw (15) to (14);
		\draw (16.center) to (7);
		\draw (7) to (14);
	\end{pgfonlayer}
\end{tikzpicture}
\eq{mono.ftwo}
\begin{tikzpicture}
	\begin{pgfonlayer}{nodelayer}
		\node [style=none] (1) at (2, 2.25) {};
		\node [style=none] (2) at (2, 0.75) {};
		\node [style=none] (3) at (1.5, 2.25) {};
		\node [style=none] (4) at (1.5, 0.75) {};
		\node [style=none] (6) at (0.5, 2.25) {};
		\node [style=none] (10) at (1, 1.5) {$\cdots$};
		\node [style=none] (11) at (1, 1.75) {$n$};
		\node [style=none] (13) at (0, 2.25) {};
		\node [style=X] (15) at (0, 1) {};
		\node [style=none] (16) at (0.5, 0.75) {};
	\end{pgfonlayer}
	\begin{pgfonlayer}{edgelayer}
		\draw (16.center) to (6.center);
		\draw (13.center) to (15);
		\draw (4.center) to (3.center);
		\draw (1.center) to (2.center);
	\end{pgfonlayer}
\end{tikzpicture}
$$
\end{definition}

\begin{lemma}
\label{lem:injand}
$\inj(\f_2)$ is a presentation for the prop $(\inj(\FSets_2),\times)$.
\end{lemma}


The pushout of a diagram of sets and functions $2^n \xleftarrowtail{} 2^k \xrightarrowtail{} 2^m$ is not always a power of 2.  Therefore, one should not expect to construct categories of partial isomorphisms via a distributive law of on $\inj(\f_2)\otimes_{\Iso(\f_2)} \inj(\f_2)^\op$. Instead one must add all of the nontrivial subobjects to the constituent props forming the distributive law; as opposed to the affine case, there are more than one such subobjects which arise in this way.

\begin{definition}
Consider the pro $\sub_2$ generated by endomorphisms such that for any $n$, $\sub_2(n,n)$ is the set described by all $n$-variable polynomials over $\F_2$.  Denote such a generator by a box with $n$ inputs and $n$ outputs labelled by the corresponding polynomial.

We require that the following equations hold so that
$$\forall n,m \in \N, p,r \in \F_2[x_1,\ldots, x_n],  q \in \F_2[x_{n+1},\ldots, x_{n+m}]:\hspace*{1cm}
\begin{tikzpicture}
	\begin{pgfonlayer}{nodelayer}
		\node [style=none] (0) at (3, 2.25) {};
		\node [style=none] (1) at (3, 3.25) {};
		\node [style=map] (2) at (3, 2.75) {$1$};
		\node [style=none] (6) at (3, 3.5) {$n$};
		\node [style=none] (8) at (3, 2) {$n$};
	\end{pgfonlayer}
	\begin{pgfonlayer}{edgelayer}
		\draw (0.center) to (1.center);
	\end{pgfonlayer}
\end{tikzpicture}
\eq{sub.one}
\begin{tikzpicture}
	\begin{pgfonlayer}{nodelayer}
		\node [style=none] (0) at (3, 2.25) {};
		\node [style=none] (1) at (3, 3.25) {};
		\node [style=none] (6) at (3, 3.5) {$n$};
		\node [style=none] (8) at (3, 2) {$n$};
	\end{pgfonlayer}
	\begin{pgfonlayer}{edgelayer}
		\draw (0.center) to (1.center);
	\end{pgfonlayer}
\end{tikzpicture}
\hspace*{,5cm}
\begin{tikzpicture}
	\begin{pgfonlayer}{nodelayer}
		\node [style=none] (0) at (5, 2) {};
		\node [style=none] (1) at (5, 4) {};
		\node [style=map] (2) at (5, 2.5) {$r$};
		\node [style=map] (3) at (5, 3.5) {$p$};
		\node [style=none] (4) at (5, 4.25) {$n$};
		\node [style=none] (5) at (5, 1.75) {$n$};
	\end{pgfonlayer}
	\begin{pgfonlayer}{edgelayer}
		\draw (0.center) to (1.center);
	\end{pgfonlayer}
\end{tikzpicture}
\eq{sub.two}
\begin{tikzpicture}
	\begin{pgfonlayer}{nodelayer}
		\node [style=none] (0) at (5, 2.25) {};
		\node [style=none] (1) at (5, 4.25) {};
		\node [style=map] (2) at (5, 3.25) {$p+r+pr$};
		\node [style=none] (4) at (5, 4.5) {$n$};
		\node [style=none] (5) at (5, 2) {$n$};
	\end{pgfonlayer}
	\begin{pgfonlayer}{edgelayer}
		\draw (0.center) to (1.center);
	\end{pgfonlayer}
\end{tikzpicture}
\hspace*{.5cm}
\begin{tikzpicture}
	\begin{pgfonlayer}{nodelayer}
		\node [style=none] (0) at (2.4, 2.25) {};
		\node [style=none] (1) at (2.4, 3.25) {};
		\node [style=map] (2) at (2.4, 2.75) {$p$};
		\node [style=none] (3) at (3, 3.25) {};
		\node [style=none] (4) at (3, 2.25) {};
		\node [style=map] (5) at (3, 2.75) {$q$};
		\node [style=none] (6) at (2.4, 3.5) {$n$};
		\node [style=none] (7) at (3, 3.5) {$m$};
		\node [style=none] (8) at (2.4, 2) {$n$};
		\node [style=none] (9) at (3, 2) {$m$};
	\end{pgfonlayer}
	\begin{pgfonlayer}{edgelayer}
		\draw (0.center) to (1.center);
		\draw (3.center) to (4.center);
	\end{pgfonlayer}
\end{tikzpicture}
\eq{sub.three}
\begin{tikzpicture}
	\begin{pgfonlayer}{nodelayer}
		\node [style=none] (0) at (2.5, 2.25) {};
		\node [style=none] (1) at (2.5, 3.25) {};
		\node [style=map] (2) at (2.75, 2.75) {$p\cdot q$};
		\node [style=none] (3) at (3, 3.25) {};
		\node [style=none] (4) at (3, 2.25) {};
		\node [style=none] (6) at (2.5, 3.5) {$n$};
		\node [style=none] (7) at (3, 3.5) {$m$};
		\node [style=none] (8) at (2.5, 2) {$n$};
		\node [style=none] (9) at (3, 2) {$m$};
	\end{pgfonlayer}
	\begin{pgfonlayer}{edgelayer}
		\draw (0.center) to (1.center);
		\draw (3.center) to (4.center);
	\end{pgfonlayer}
\end{tikzpicture}
$$
As well as, for all $n$, the equations of the quotient rings  $\F_2[x_1,\ldots, x_n]/\langle x_1^2-x_1,\ldots, x_n^2-x_n \rangle$.

\end{definition}

\begin{lemma}
\label{lem:sub}
$\sub_2$ is a presentation for the pro of symmetric spans of monic functions, ie spans of the following form $2^n \xleftarrow{e} k \xrightarrow{e}2^n$, for all $n,k \in \N$ and monics $e$.
\end{lemma}

\begin{proof}
Each polynomial  $p \in \F_2[x_1,\ldots, x_n]/\langle x_1^2-x_1,\ldots, x_n^2-x_n \rangle$ corresponds to a function $\ev_p:\Z_2^n \to \Z_2$ given by evaluation.  Let $k = |\ev^{-1}(1)|$, then there chose a function $f_p:k \rightarrowtail 2^n$ picking out all the solutions which evaluate to $1$. The functor from $\sub_2$ to this subcategory spans takes polynomials $p \mapsto (2^n \xleftarrowtail {f_p} k \xrightarrowtail {f_p} 2^n)$.  Any two two spans induced by the same polynomial are isomorphic, so this is actually well defined.  It is clearly an isomorphism on objects, and it can easily be shown to be a monoidal functor.

The fullness is easy and the faithfulness comes from the fact that we can reduce every map to a polynomial and then reduce the polynomial to algebraic normal form.

\end{proof}

%

\begin{definition}
Let $\sub\Iso\f_2$ be the prop generated by a distributive law of pros:
$$
\sub_2 \otimes \Iso(\f_2);
$$
$$
 \forall n,m,k \in \N, \forall p \in \F_2[x_1,\ldots, x_{n+2+m}],
q \in \F_2[x_1,\ldots,x_{n+m+1+k}],
r \in \F_2[x_1,\ldots, x_n]:
$$
$$
\begin{tikzpicture}
	\begin{pgfonlayer}{nodelayer}
		\node [style=map] (0) at (5, 3.5) {$p(x_1,\ldots, x_n, x_{n+1}, x_{n+2}, x_{n+3},\ldots, x_{n+2+m})$};
		\node [style=none] (1) at (4.5, 4.25) {};
		\node [style=none] (2) at (5.5, 4.25) {};
		\node [style=none] (3) at (4.75, 4.25) {};
		\node [style=none] (4) at (5.25, 4.25) {};
		\node [style=none] (5) at (4.5, 2.25) {};
		\node [style=none] (6) at (5.5, 2.25) {};
		\node [style=none] (7) at (4.75, 2.75) {};
		\node [style=none] (8) at (5.25, 2.75) {};
		\node [style=none] (9) at (4.5, 4.5) {$n$};
		\node [style=none] (10) at (4.5, 2) {$n$};
		\node [style=none] (11) at (5.5, 4.5) {$m$};
		\node [style=none] (12) at (5.5, 2) {$m$};
		\node [style=none] (13) at (5.25, 2.25) {};
		\node [style=none] (14) at (4.75, 2.25) {};
	\end{pgfonlayer}
	\begin{pgfonlayer}{edgelayer}
		\draw [in=120, out=-90] (1.center) to (0);
		\draw [in=-90, out=60] (0) to (2.center);
		\draw [in=75, out=-90, looseness=0.75] (4.center) to (0);
		\draw [in=-90, out=105, looseness=0.75] (0) to (3.center);
		\draw [in=300, out=90] (6.center) to (0);
		\draw [in=90, out=-75, looseness=0.75] (0) to (8.center);
		\draw [in=255, out=90, looseness=0.75] (7.center) to (0);
		\draw [in=90, out=-120] (0) to (5.center);
		\draw [in=270, out=90] (13.center) to (7.center);
		\draw [in=270, out=90] (14.center) to (8.center);
	\end{pgfonlayer}
\end{tikzpicture}
\eq{subiso.one}
\begin{tikzpicture}
	\begin{pgfonlayer}{nodelayer}
		\node [style=map] (0) at (5, 3.25) {$p(x_1,\ldots, x_n, x_{n+2}, x_{n+1}, x_{n+3},\ldots, x_{n+2+m})$};
		\node [style=none] (1) at (4.5, 2.5) {};
		\node [style=none] (2) at (5.5, 2.5) {};
		\node [style=none] (3) at (4.75, 2.5) {};
		\node [style=none] (4) at (5.25, 2.5) {};
		\node [style=none] (5) at (4.5, 4.5) {};
		\node [style=none] (6) at (5.5, 4.5) {};
		\node [style=none] (7) at (4.75, 4) {};
		\node [style=none] (8) at (5.25, 4) {};
		\node [style=none] (9) at (4.5, 2.25) {$n$};
		\node [style=none] (10) at (4.5, 4.75) {$n$};
		\node [style=none] (11) at (5.5, 2.25) {$m$};
		\node [style=none] (12) at (5.5, 4.75) {$m$};
		\node [style=none] (13) at (5.25, 4.5) {};
		\node [style=none] (14) at (4.75, 4.5) {};
	\end{pgfonlayer}
	\begin{pgfonlayer}{edgelayer}
		\draw [in=-120, out=90] (1.center) to (0);
		\draw [in=90, out=-60] (0) to (2.center);
		\draw [in=-75, out=90, looseness=0.75] (4.center) to (0);
		\draw [in=90, out=-105, looseness=0.75] (0) to (3.center);
		\draw [in=-300, out=-90] (6.center) to (0);
		\draw [in=-90, out=75, looseness=0.75] (0) to (8.center);
		\draw [in=-255, out=-90, looseness=0.75] (7.center) to (0);
		\draw [in=-90, out=120] (0) to (5.center);
		\draw [in=-270, out=-90] (13.center) to (7.center);
		\draw [in=-270, out=-90] (14.center) to (8.center);
	\end{pgfonlayer}
\end{tikzpicture}
$$
$$
\begin{tikzpicture}
	\begin{pgfonlayer}{nodelayer}
		\node [style=map] (0) at (5, 3.5) {$q(x_1,\ldots, x_{n+m+1+k})$};
		\node [style=none] (1) at (3.75, 4) {};
		\node [style=none] (2) at (6.25, 4) {};
		\node [style=none] (3) at (4.25, 4) {};
		\node [style=none] (4) at (5.75, 4) {};
		\node [style=none] (5) at (3.75, 3) {};
		\node [style=none] (6) at (6.25, 3) {};
		\node [style=none] (7) at (4.25, 3) {};
		\node [style=none] (8) at (5.75, 3) {};
		\node [style=none] (9) at (3.75, 2.25) {$n$};
		\node [style=none] (10) at (4.25, 2.5) {};
		\node [style=none] (11) at (5.75, 2.5) {};
		\node [style=dot] (12) at (4.25, 2.75) {};
		\node [style=oplus] (13) at (5.75, 2.75) {};
		\node [style=none] (14) at (3.75, 2.5) {};
		\node [style=none] (15) at (6.25, 2.5) {};
		\node [style=none] (16) at (5.25, 4) {};
		\node [style=dot] (17) at (5.25, 2.75) {};
		\node [style=none] (18) at (5.25, 2.5) {};
		\node [style=none] (19) at (6.25, 2.25) {$k$};
		\node [style=none] (20) at (4.75, 2.5) {$m$};
		\node [style=none] (21) at (3.75, 5) {$n$};
		\node [style=none] (22) at (6.25, 5) {$k$};
		\node [style=none] (23) at (4.75, 4) {$m$};
		\node [style=none] (24) at (4.25, 4.5) {};
		\node [style=none] (25) at (5.75, 4.5) {};
		\node [style=none] (26) at (3.75, 4.5) {};
		\node [style=none] (27) at (6.25, 4.5) {};
		\node [style=none] (28) at (5.25, 4.5) {};
	\end{pgfonlayer}
	\begin{pgfonlayer}{edgelayer}
		\draw [in=270, out=90] (10.center) to (7.center);
		\draw [in=270, out=90] (11.center) to (8.center);
		\draw (15.center) to (6.center);
		\draw (14.center) to (5.center);
		\draw (13) to (17);
		\draw [style=dotted] (17) to (12);
		\draw (18.center) to (17);
		\draw (1.center) to (26.center);
		\draw (3.center) to (24.center);
		\draw (16.center) to (28.center);
		\draw (4.center) to (25.center);
		\draw (2.center) to (27.center);
		\draw (1.center) to (5.center);
		\draw (12) to (3.center);
		\draw (16.center) to (17);
		\draw (13) to (4.center);
		\draw (2.center) to (6.center);
	\end{pgfonlayer}
\end{tikzpicture}
\eq{subiso.two}
\begin{tikzpicture}
	\begin{pgfonlayer}{nodelayer}
		\node [style=map] (0) at (5, 3) {$q(x_1,\ldots, x_{n+m}, (x_{n+1}\ldots x_{n+m-1})+x_{n+m+1}, x_{n+m+2}, \ldots, x_{n+m+1+k})$};
		\node [style=none] (1) at (3.75, 2.5) {};
		\node [style=none] (2) at (6.25, 2.5) {};
		\node [style=none] (3) at (4.25, 2.5) {};
		\node [style=none] (4) at (5.75, 2.5) {};
		\node [style=none] (5) at (3.75, 3.5) {};
		\node [style=none] (6) at (6.25, 3.5) {};
		\node [style=none] (7) at (4.25, 3.5) {};
		\node [style=none] (8) at (5.75, 3.5) {};
		\node [style=none] (9) at (3.75, 4.25) {$n$};
		\node [style=none] (10) at (4.25, 4) {};
		\node [style=none] (11) at (5.75, 4) {};
		\node [style=dot] (12) at (4.25, 3.75) {};
		\node [style=oplus] (13) at (5.75, 3.75) {};
		\node [style=none] (14) at (3.75, 4) {};
		\node [style=none] (15) at (6.25, 4) {};
		\node [style=none] (16) at (5.25, 2.5) {};
		\node [style=dot] (17) at (5.25, 3.75) {};
		\node [style=none] (18) at (5.25, 4) {};
		\node [style=none] (19) at (6.25, 4.25) {$k$};
		\node [style=none] (20) at (4.75, 4) {$m$};
		\node [style=none] (21) at (3.75, 1.75) {$n$};
		\node [style=none] (22) at (6.25, 1.75) {$k$};
		\node [style=none] (23) at (4.75, 2.5) {$m$};
		\node [style=none] (24) at (4.25, 2) {};
		\node [style=none] (25) at (5.75, 2) {};
		\node [style=none] (26) at (3.75, 2) {};
		\node [style=none] (27) at (6.25, 2) {};
		\node [style=none] (28) at (5.25, 2) {};
	\end{pgfonlayer}
	\begin{pgfonlayer}{edgelayer}
		\draw [in=-270, out=-90] (10.center) to (7.center);
		\draw [in=-270, out=-90] (11.center) to (8.center);
		\draw (15.center) to (6.center);
		\draw (14.center) to (5.center);
		\draw (13) to (17);
		\draw [style=dotted] (17) to (12);
		\draw (18.center) to (17);
		\draw (26.center) to (1.center);
		\draw (24.center) to (3.center);
		\draw (28.center) to (16.center);
		\draw (25.center) to (4.center);
		\draw (27.center) to (2.center);
		\draw (1.center) to (5.center);
		\draw (3.center) to (12);
		\draw (16.center) to (17);
		\draw (4.center) to (13);
		\draw (2.center) to (6.center);
	\end{pgfonlayer}
\end{tikzpicture}\hspace*{.5cm}
\begin{tikzpicture}
	\begin{pgfonlayer}{nodelayer}
		\node [style=none] (5) at (3.25, 4.25) {};
		\node [style=none] (6) at (3.25, 3.25) {};
		\node [style=map] (9) at (3.25, 3.75) {$r$};
		\node [style=none] (10) at (3.25, 2.25) {};
		\node [style=none] (11) at (3.75, 4.25) {};
		\node [style=none] (12) at (3.75, 2.25) {};
		\node [style=none] (13) at (3.75, 3.25) {};
	\end{pgfonlayer}
	\begin{pgfonlayer}{edgelayer}
		\draw (5.center) to (6.center);
		\draw [in=90, out=-90] (13.center) to (10.center);
		\draw [in=-90, out=90] (12.center) to (6.center);
		\draw (13.center) to (11.center);
	\end{pgfonlayer}
\end{tikzpicture}
\eq{subiso.three}
\begin{tikzpicture}
	\begin{pgfonlayer}{nodelayer}
		\node [style=none] (5) at (3.75, 2.25) {};
		\node [style=none] (6) at (3.75, 3.25) {};
		\node [style=map] (9) at (3.75, 2.75) {$r$};
		\node [style=none] (10) at (3.75, 4.25) {};
		\node [style=none] (11) at (3.25, 2.25) {};
		\node [style=none] (12) at (3.25, 4.25) {};
		\node [style=none] (13) at (3.25, 3.25) {};
	\end{pgfonlayer}
	\begin{pgfonlayer}{edgelayer}
		\draw (5.center) to (6.center);
		\draw [in=270, out=90] (13.center) to (10.center);
		\draw [in=90, out=-90] (12.center) to (6.center);
		\draw (13.center) to (11.center);
	\end{pgfonlayer}
\end{tikzpicture}
$$

\end{definition}

\begin{lemma}
\label{lem:subiso}
$\sub\Iso\f_2$ is a presentation for the subcategory of $(\Span(\FSets),\times)$ generated by spans of the form $2^n \xleftarrowtail{e} k \xrightarrowtail {e} 2^m \xrightarrow[\cong]{f} 2^m$, for all $n,m k \in \N$ and all isomorphisms $f$ and monics $e$.
\end{lemma}

\begin{proof}
The obvious functor is clearly monoidal. Moreover, it is full by construction.
For the faithfulness, take two maps $f$ and $g$ in $\sub\Iso\f_2$.  Then one can just push everything to the end and then use the decidability of equality on both factors of the distributive law to show that they are equal.
\end{proof}


\begin{definition}
Consider the prop $\sub\inj\f_2$ generated by a distributive law of props:
$$
 \sub\Iso\f_2 \otimes_{\Iso(\f_2)} \inj(\f_2);
 \forall n,m \in \N, p \in \F_2[x_1,\ldots, x_{n+1+m}]:
\hspace*{.5cm}
\begin{tikzpicture}
	\begin{pgfonlayer}{nodelayer}
		\node [style=none] (0) at (1.5, 3.5) {};
		\node [style=map] (1) at (2.5, 2.75) {$p(x_1,\ldots,x_{n+1+m})$};
		\node [style=none] (2) at (1.5, 3.75) {$n$};
		\node [style=none] (3) at (3.5, 3.5) {};
		\node [style=none] (4) at (3.5, 3.75) {$m$};
		\node [style=none] (5) at (1.5, 1.75) {};
		\node [style=none] (6) at (3.5, 1.75) {};
		\node [style=none] (7) at (1.5, 1.5) {$n$};
		\node [style=none] (8) at (3.5, 1.5) {$m$};
		\node [style=X] (9) at (2.5, 2) {};
		\node [style=none] (10) at (2.5, 3.5) {};
	\end{pgfonlayer}
	\begin{pgfonlayer}{edgelayer}
		\draw (0.center) to (5.center);
		\draw (3.center) to (6.center);
		\draw (10.center) to (9);
	\end{pgfonlayer}
\end{tikzpicture}
\eq{subinj}
\begin{tikzpicture}
	\begin{pgfonlayer}{nodelayer}
		\node [style=none] (0) at (1.5, 3.75) {};
		\node [style=map] (1) at (2.5, 2.75) {$p(x_1,\ldots,x_n,0,x_{n+2},\ldots,x_{n+1+m})$};
		\node [style=none] (2) at (1.5, 4) {$n$};
		\node [style=none] (3) at (3.5, 3.75) {};
		\node [style=none] (4) at (3.5, 4) {$m$};
		\node [style=none] (5) at (1.5, 2) {};
		\node [style=none] (6) at (3.5, 2) {};
		\node [style=none] (7) at (1.5, 1.75) {$n$};
		\node [style=none] (8) at (3.5, 1.75) {$m$};
		\node [style=X] (9) at (2.5, 3.4) {};
		\node [style=none] (10) at (2.5, 3.75) {};
	\end{pgfonlayer}
	\begin{pgfonlayer}{edgelayer}
		\draw (0.center) to (5.center);
		\draw (3.center) to (6.center);
		\draw (10.center) to (9);
	\end{pgfonlayer}
\end{tikzpicture}
$$
\end{definition}

\begin{lemma}
\label{lem:subinj}
$\sub\inj\f_2$ is a presentation for the subcategory of $(\Span(\FSets),\times)$ generated by spans of the form $2^n \xleftarrowtail{e} k \xrightarrowtail{e} 2^n \xrightarrowtail{e'} 2^{m}$ for all $n,m,k \in \N$ and all monics $e,e'$.
\end{lemma}

The proof is completely analogous to as in the case of $\sub\iso\f_2$.

Any $n$ variable polynomial $p$ can be interpreted as a span of monics via the oracle $\mathcal{O}_p$, where the value of the target wire is restricted to have the value $0$.  Each such polynomial corresponds to a subobject, which complicates the matter further than in the affine case.

\begin{definition}
Consider the prop $\pr\iso\f_2$ given by the distributive law of props:

$$
\sub\inj\f_2^\op \otimes_{\sub\Iso\f_2} \sub\inj\f_2;
\begin{tikzpicture}
	\begin{pgfonlayer}{nodelayer}
		\node [style=map] (0) at (0, 0) {$\mathcal{O}_p$};
		\node [style=none] (1) at (-0.25, -0.75) {};
		\node [style=none] (2) at (-0.25, 0.75) {};
		\node [style=none] (3) at (-0.25, 1) {};
		\node [style=none] (4) at (-0.25, -1) {};
		\node [style=X] (5) at (0.25, 0.75) {};
		\node [style=X] (6) at (0.25, -0.75) {};
	\end{pgfonlayer}
	\begin{pgfonlayer}{edgelayer}
		\draw (3.center) to (2.center);
		\draw (4.center) to (1.center);
		\draw [in=-60, out=90] (6) to (0);
		\draw [in=-90, out=60] (0) to (5);
		\draw [in=120, out=-90] (2.center) to (0);
		\draw [in=90, out=-120, looseness=1.25] (0) to (1.center);
	\end{pgfonlayer}
\end{tikzpicture}
\eq{oracle}
\begin{tikzpicture}
	\begin{pgfonlayer}{nodelayer}
		\node [style=map] (0) at (-0.25, 0) {$p$};
		\node [style=none] (1) at (-0.25, -0.75) {};
		\node [style=none] (2) at (-0.25, 0.75) {};
		\node [style=none] (3) at (-0.25, 1) {};
		\node [style=none] (4) at (-0.25, -1) {};
	\end{pgfonlayer}
	\begin{pgfonlayer}{edgelayer}
		\draw (3.center) to (2.center);
		\draw (4.center) to (1.center);
		\draw (2.center) to (0);
		\draw (0) to (1.center);
	\end{pgfonlayer}
\end{tikzpicture}
$$

\end{definition}



\begin{lemma}
\label{lem:parisof}
$\pr\iso \f_2$ is a presentation for the full subcategory $(\FPinj_2,\times)$ of $(\ParIso(\FSets),\times)$ with objects powers of two.
\end{lemma}

Unlike the previous lemmas, this is not dependant on a complete presentation for the isomorphisms.
The proof is a consequence of \cite[Thm 7.6.14]{cole} where they give a finite, complete presentation for this category.  The identities up to this point are equivalent to this finite presentation, whether or not the conjectured presentation for the isomorphisms is complete.

\begin{definition}

Consider the prop $\pr\f_2$ given by the  pushout of the following diagram of props, given by adding a counit to the diagonal map:
$$\pr\iso\f_2 \leftarrow \surj^\op \rightarrow \cm^\op$$

\end{definition}

\begin{lemma}
\label{lem:parand}
$\pr\f_2$ is a presentation for the the full subcategory $(\FPar_2,\times)$ of $(\Par(\FSets),\times)$ with objects powers of two.
\end{lemma}

The proof is contained in \S \ref{proof:lem:parand}.

%

\begin{definition}
Let $\sp\f_2$ denote the pushout of the diagram of props:
$$
\pr\f_2^\op\leftarrow \pr\iso\f_2 \rightarrow \pr\f_2
$$
\end{definition}

\begin{lemma}\cite{zxa}
\label{lem:spanand}
$\sp\f_2$ is a presentation for  the full subcategory $(\FSpan_2,\times)$ of $(\Span(\FSets),\times)$ with objects powers of two.
\end{lemma}

See \S \ref{proof:lem:spanand} for the proof.  There is a particularly elegant finite presentation 
which  is much more ZX-flavoured, dubbed $\ZXA$.  Based on a similar observation to one made in \cite{bruni}, the author of \cite{zxa} remarks that this category is equivalent as a prop to the full subcategory of natural number matrices where the objects are powers of 2.

As remarked in \cite{zxa}, this category is equivalent to the "natural number H-box fragment" of the ZH-calculus.  That is to say, the prop generated by the Z and X spiders (coresponding to (co)copying and (co)addition) along with H-boxes which have values restricted to be natural numbers.  The interpretation of these H-boxes is given in \cite[Fig. 5]{zxa}, which we restate for reference:
$$
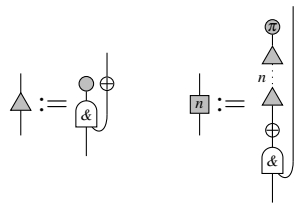
\begin{tikzpicture}
	\begin{pgfonlayer}{nodelayer}
		\node [style=none] (0) at (0.75, 1.25) {};
		\node [style=none] (1) at (0.75, 0.5) {};
		\node [style=none] (2) at (0.75, 2) {};
		\node [style=triflip] (3) at (0.75, 1.25) {};
	\end{pgfonlayer}
	\begin{pgfonlayer}{edgelayer}
		\draw (2.center) to (0.center);
		\draw (0.center) to (1.center);
	\end{pgfonlayer}
\end{tikzpicture}
:=
\begin{tikzpicture}
	\begin{pgfonlayer}{nodelayer}
		\node [style=none] (0) at (0.75, 1.5) {};
		\node [style=none] (1) at (1.25, 1.5) {};
		\node [style=X] (2) at (0.75, 2.25) {};
		\node [style=none] (3) at (0.75, 0.5) {};
		\node [style=andin] (4) at (0.75, 1.5) {};
		\node [style=none] (5) at (1.25, 3) {};
		\node [style=oplus] (6) at (1.25, 2.25) {};
	\end{pgfonlayer}
	\begin{pgfonlayer}{edgelayer}
		\draw [style=simple] (2) to (0.center);
		\draw [style=simple, in=-75, out=-90, looseness=2.75] (1.center) to (0.center);
		\draw [style=simple] (3.center) to (0.center);
		\draw (5.center) to (6);
		\draw (6) to (1.center);
	\end{pgfonlayer}
\end{tikzpicture}
\hspace*{1cm}
\begin{tikzpicture}
	\begin{pgfonlayer}{nodelayer}
		\node [style=H] (0) at (1, 1.25) {$n$};
		\node [style=none] (1) at (1, 0.5) {};
		\node [style=none] (2) at (1, 2) {};
	\end{pgfonlayer}
	\begin{pgfonlayer}{edgelayer}
		\draw (2.center) to (0);
		\draw (0) to (1.center);
	\end{pgfonlayer}
\end{tikzpicture}
:=
\begin{tikzpicture}
	\begin{pgfonlayer}{nodelayer}
		\node [style=none] (0) at (0.75, 1.5) {};
		\node [style=none] (1) at (1.25, 1.5) {};
		\node [style=X] (2) at (0.75, 4.75) {$\pi$};
		\node [style=none] (3) at (0.75, 0.5) {};
		\node [style=andin] (4) at (0.75, 1.5) {};
		\node [style=none] (5) at (1.25, 5.25) {};
		\node [style=triflip] (6) at (0.75, 4) {};
		\node [style=triflip] (7) at (0.75, 3) {};
		\node [style=none] (8) at (0.5, 3.5) {$n$};
		\node [style=oplus] (9) at (0.75, 2.25) {};
	\end{pgfonlayer}
	\begin{pgfonlayer}{edgelayer}
		\draw [style=simple, in=-75, out=-90, looseness=2.75] (1.center) to (0.center);
		\draw [style=simple] (3.center) to (0.center);
		\draw [style=dotted] (6) to (7);
		\draw (6) to (2);
		\draw (7) to (0.center);
		\draw [style=simple] (1.center) to (5.center);
	\end{pgfonlayer}
\end{tikzpicture}
$$
Note that arbitrary arity H-boxes can be obtained by composition with and gates.

\section{Conclusion and future work}

In this paper, we have devised a method to give modular presentations for full subcategories of categories of spans; albeit, this method is proven to work in full generality. In all three cases we we have considered, there is a fully faithful symmetric monoidal functor $F:\X\to\Y$ between a prop $\X$ and a symmetric monoidal category $\Y$.  And the categories which we eventually build up to are the full subcategories of $\Span(\Y)$ with objects in $F(\X)$.  In other words, these are structured span categories ${}_F\Span(\Y)$, as considered in \cite{structured} (or structure cospan categories ${}_{F^\op}\Csp(\Y^\op)$).  We build up to ${}_F\Span(\Y)$ first by presenting the isos and monics in $\Y$ in terms of a symmetric monoidal theory.  Then we consider the subcategory of ${}_F\Span(\Y)$ generated by monic spans spans of the form $FX\xleftarrowtail{e} Y \xrightarrowtail{e} FX$, corresponding to the new subobjects created in pullback of maps in $F(\X)$: presenting these as monoidal theories.  Then we add these subobjects to the isos and  monos by distributive law and again present these in terms of symmetric monoidal theories.  After doing so, we are able to construct a distributive law between the monics and co-monics in ${}_F\Span(\Y)$ up to isomorphisms with subobjects adjoined to all three props (this is a crucial step because in the nonlinear case, not all pullbacks exist, so one can not construct a distributive law over the isomorphisms).  We then observe that the prop generated by such a distributive law is a discrete inverse category (as defined in \cite[Def. 4.3.1]{giles}), thus one can complete it to a discrete Cartesian restriction category by adding counits to the codiagonal map (as observed in \cite[Lem. 3.5]{zxa}).  Finally, one can glue together the discrete inverse category to its opposite category up to the shared discrete inverse category to obtain ${}_F\Span(\Y)$.

This method is quite generic, and it would be useful to establish some criteria for when it can be applied categories of structured spans.
There are various ways in which this method could potentially be employed to give a modular presentation of such a category of structured spans.
Possibly the easiest such class of examples would be that given by the functor $(\Aff\Mat(k),+)\to (\Aff\Vect(k),+)$ , for an arbitrary field (or maybe even PID $k$).  More difficult, would be the class of examples given by $\FinOrd_p \to (\FinOrd,\times)$, for arbitrary prime $p>2$.  As opposed to in the case of $p=2$, the presentation for the prop of subobjects would be less simple; only in this case can systems of multivariate polynomials always be reduced to a single polynomial.  To produce a normal form in the more general case, we expect that one would have to employ the use of Gr\"obner bases.  Another model which one could pursue is that given by the functor from free, finitely generated commutative semigroups to additive monoids.  A presentation is given for the corresponding category of relations  in \cite[\S 3.3]{piedeleu}; however, it is not given a modular treatment.  This would potentially be useful because of applications in concurrency theory.

Other directions would be to try to add more phases in a modular fashion. 
Perhaps the work of \cite{duncan} could help add more phases to this picture in a modular fashion.

{\bf Acknowledgements}

The author thank Aleks Kissinger for useful discussions.
%
%
%
%
%
%
%
%
%
%

%
%

\nocite{ih}
\nocite{zx}
\nocite{zh}
\nocite{tof}

\bibliography{propsnew.bib}
\bibliographystyle{plain}

\appendix

\section{Proofs}
\subsection{Section \ref{sec:one}}

\subsubsection{Remark on Definition \ref{def:pariso:cb}}

\begin{remark}
\label{rem:pariso:cb}
This is actually a distributive law because the only only seemingly nontrivial situation arises when controlled not gates are sandwiched by black units/counits on their target wires.  However the following identity holds by induction on the number of controlled not gates.   For the base case of $n=0$, this follows from the bone law which we added to the distributive law.  For $n>1$, we have the following situation:
$$
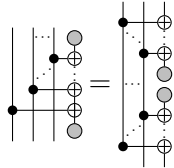
\begin{tikzpicture}
	\begin{pgfonlayer}{nodelayer}
		\node [style=X] (0) at (1.5, -0.25) {};
		\node [style=oplus] (1) at (1.5, -0.75) {};
		\node [style=oplus] (2) at (1.5, -1.5) {};
		\node [style=dot] (3) at (1, -0.75) {};
		\node [style=dot] (4) at (0.5, -1.5) {};
		\node [style=none] (5) at (1, 0) {};
		\node [style=none] (6) at (0.5, 0) {};
		\node [style=none] (7) at (0.75, -1) {$\iddots$};
		\node [style=none] (8) at (1.5, -1) {$\vdots$};
		\node [style=X] (9) at (1.5, -2.5) {};
		\node [style=none] (10) at (0.5, -2.75) {};
		\node [style=none] (11) at (1, -2.75) {};
		\node [style=none] (12) at (0.75, -0.25) {$\cdots$};
		\node [style=oplus] (13) at (1.5, -2) {};
		\node [style=dot] (14) at (0, -2) {};
		\node [style=none] (15) at (0, -2.75) {};
		\node [style=none] (16) at (0, 0) {};
		\node [style=none] (18) at (1.25, -0.75) {};
	\end{pgfonlayer}
	\begin{pgfonlayer}{edgelayer}
		\draw (0) to (1);
		\draw (1) to (3);
		\draw (5.center) to (3);
		\draw (6.center) to (4);
		\draw (4) to (2);
		\draw (4) to (10.center);
		\draw (11.center) to (3);
		\draw (9) to (2);
		\draw (14) to (13);
		\draw (15.center) to (14);
		\draw (14) to (16.center);
	\end{pgfonlayer}
\end{tikzpicture}
=
\begin{tikzpicture}
	\begin{pgfonlayer}{nodelayer}
		\node [style=X] (19) at (3.5, -1.25) {};
		\node [style=none] (24) at (3.5, 0.5) {};
		\node [style=none] (25) at (3, 0.5) {};
		\node [style=X] (28) at (3.5, -1.75) {};
		\node [style=none] (29) at (3, -3.5) {};
		\node [style=none] (30) at (3.5, -3.5) {};
		\node [style=none] (34) at (2.5, -3.5) {};
		\node [style=none] (35) at (2.5, 0.5) {};
		\node [style=oplus] (36) at (3.5, -0.75) {};
		\node [style=oplus] (37) at (3.5, 0) {};
		\node [style=oplus] (38) at (3.5, -3) {};
		\node [style=oplus] (39) at (3.5, -2.25) {};
		\node [style=dot] (40) at (2.5, -3) {};
		\node [style=dot] (41) at (3, -2.25) {};
		\node [style=dot] (42) at (3, -0.75) {};
		\node [style=dot] (43) at (2.5, 0) {};
		\node [style=none] (44) at (3.5, -0.25) {$\vdots$};
		\node [style=none] (45) at (3.5, -2.5) {$\vdots$};
		\node [style=none] (46) at (2.75, -1.5) {$\cdots$};
		\node [style=none] (47) at (2.75, -2.5) {$\iddots$};
		\node [style=none] (48) at (2.75, -0.25) {$\ddots$};
	\end{pgfonlayer}
	\begin{pgfonlayer}{edgelayer}
		\draw (24.center) to (37);
		\draw (36) to (19);
		\draw (28) to (39);
		\draw (38) to (30.center);
		\draw (29.center) to (41);
		\draw (41) to (42);
		\draw (42) to (25.center);
		\draw (35.center) to (43);
		\draw (43) to (40);
		\draw (40) to (38);
		\draw (37) to (43);
		\draw (42) to (36);
		\draw (39) to (41);
	\end{pgfonlayer}
\end{tikzpicture}
$$
For $n=1$, that is:
$$
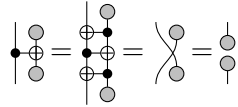
\begin{tikzpicture}
	\begin{pgfonlayer}{nodelayer}
		\node [style=X] (0) at (4, 2) {};
		\node [style=X] (1) at (4, 1) {};
		\node [style=oplus] (2) at (4, 1.5) {};
		\node [style=dot] (3) at (3.5, 1.5) {};
		\node [style=none] (4) at (3.5, 2.25) {};
		\node [style=none] (5) at (3.5, 0.75) {};
	\end{pgfonlayer}
	\begin{pgfonlayer}{edgelayer}
		\draw (1) to (0);
		\draw (4.center) to (5.center);
		\draw (3) to (2);
	\end{pgfonlayer}
\end{tikzpicture}
=
\begin{tikzpicture}
	\begin{pgfonlayer}{nodelayer}
		\node [style=X] (0) at (4, 2.5) {};
		\node [style=X] (1) at (4, 0.5) {};
		\node [style=oplus] (2) at (4, 1.5) {};
		\node [style=dot] (3) at (3.5, 1.5) {};
		\node [style=none] (4) at (3.5, 2.75) {};
		\node [style=none] (5) at (3.5, 0.25) {};
		\node [style=oplus] (6) at (3.5, 2) {};
		\node [style=dot] (7) at (4, 2) {};
		\node [style=oplus] (8) at (3.5, 1) {};
		\node [style=dot] (9) at (4, 1) {};
	\end{pgfonlayer}
	\begin{pgfonlayer}{edgelayer}
		\draw (1) to (0);
		\draw (4.center) to (5.center);
		\draw (3) to (2);
		\draw (7) to (6);
		\draw (9) to (8);
	\end{pgfonlayer}
\end{tikzpicture}
=
\begin{tikzpicture}
	\begin{pgfonlayer}{nodelayer}
		\node [style=X] (0) at (4, 2.5) {};
		\node [style=X] (1) at (4, 1.5) {};
		\node [style=none] (4) at (3.5, 2.75) {};
		\node [style=none] (5) at (3.5, 1.25) {};
	\end{pgfonlayer}
	\begin{pgfonlayer}{edgelayer}
		\draw [in=-90, out=90] (1) to (4.center);
		\draw [in=90, out=-90] (0) to (5.center);
	\end{pgfonlayer}
\end{tikzpicture}
=
\begin{tikzpicture}
	\begin{pgfonlayer}{nodelayer}
		\node [style=X] (0) at (3.5, 1.75) {};
		\node [style=X] (1) at (3.5, 2.25) {};
		\node [style=none] (4) at (3.5, 2.75) {};
		\node [style=none] (5) at (3.5, 1.25) {};
	\end{pgfonlayer}
	\begin{pgfonlayer}{edgelayer}
		\draw [in=-90, out=90] (1) to (4.center);
		\draw [in=90, out=-90] (0) to (5.center);
	\end{pgfonlayer}
\end{tikzpicture}
$$
And for the base case for $n=2$:
$$
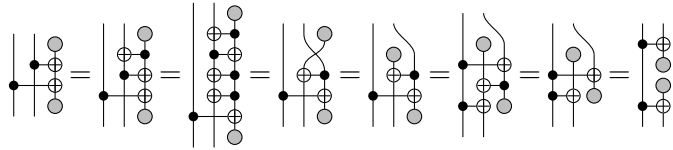
\begin{tikzpicture}
	\begin{pgfonlayer}{nodelayer}
		\node [style=X] (0) at (4.5, 2) {};
		\node [style=X] (1) at (4.5, 0.5) {};
		\node [style=oplus] (2) at (4.5, 1) {};
		\node [style=dot] (3) at (3.5, 1) {};
		\node [style=none] (4) at (3.5, 2.25) {};
		\node [style=none] (5) at (3.5, 0.25) {};
		\node [style=oplus] (6) at (4.5, 1.5) {};
		\node [style=dot] (7) at (4, 1.5) {};
		\node [style=none] (8) at (4, 2.25) {};
		\node [style=none] (9) at (4, 0.25) {};
	\end{pgfonlayer}
	\begin{pgfonlayer}{edgelayer}
		\draw (1) to (0);
		\draw (4.center) to (5.center);
		\draw (3) to (2);
		\draw (8.center) to (9.center);
		\draw (7) to (6);
	\end{pgfonlayer}
\end{tikzpicture}
=
\begin{tikzpicture}
	\begin{pgfonlayer}{nodelayer}
		\node [style=X] (0) at (4.5, 2.5) {};
		\node [style=X] (1) at (4.5, 0.5) {};
		\node [style=oplus] (2) at (4.5, 1) {};
		\node [style=dot] (3) at (3.5, 1) {};
		\node [style=none] (4) at (3.5, 2.75) {};
		\node [style=none] (5) at (3.5, 0.25) {};
		\node [style=oplus] (6) at (4.5, 1.5) {};
		\node [style=dot] (7) at (4, 1.5) {};
		\node [style=none] (8) at (4, 2.75) {};
		\node [style=none] (9) at (4, 0.25) {};
		\node [style=oplus] (10) at (4, 2) {};
		\node [style=dot] (11) at (4.5, 2) {};
	\end{pgfonlayer}
	\begin{pgfonlayer}{edgelayer}
		\draw (1) to (0);
		\draw (4.center) to (5.center);
		\draw (3) to (2);
		\draw (8.center) to (9.center);
		\draw (7) to (6);
		\draw (11) to (10);
	\end{pgfonlayer}
\end{tikzpicture}
=
\begin{tikzpicture}
	\begin{pgfonlayer}{nodelayer}
		\node [style=X] (0) at (4.5, 3.5) {};
		\node [style=X] (1) at (4.5, 0.5) {};
		\node [style=oplus] (2) at (4.5, 1) {};
		\node [style=dot] (3) at (3.5, 1) {};
		\node [style=none] (4) at (3.5, 3.75) {};
		\node [style=none] (5) at (3.5, 0.25) {};
		\node [style=oplus] (6) at (4.5, 2.5) {};
		\node [style=dot] (7) at (4, 2.5) {};
		\node [style=none] (8) at (4, 3.75) {};
		\node [style=none] (9) at (4, 0.25) {};
		\node [style=oplus] (10) at (4, 3) {};
		\node [style=dot] (11) at (4.5, 3) {};
		\node [style=oplus] (12) at (4, 2) {};
		\node [style=dot] (13) at (4.5, 2) {};
		\node [style=oplus] (14) at (4, 1.5) {};
		\node [style=dot] (15) at (4.5, 1.5) {};
	\end{pgfonlayer}
	\begin{pgfonlayer}{edgelayer}
		\draw (1) to (0);
		\draw (4.center) to (5.center);
		\draw (3) to (2);
		\draw (8.center) to (9.center);
		\draw (7) to (6);
		\draw (11) to (10);
		\draw (13) to (12);
		\draw (15) to (14);
	\end{pgfonlayer}
\end{tikzpicture}
=
\begin{tikzpicture}
	\begin{pgfonlayer}{nodelayer}
		\node [style=X] (0) at (4.5, 2.5) {};
		\node [style=X] (1) at (4.5, 0.5) {};
		\node [style=oplus] (2) at (4.5, 1) {};
		\node [style=dot] (3) at (3.5, 1) {};
		\node [style=none] (4) at (3.5, 2.75) {};
		\node [style=none] (5) at (3.5, 0.25) {};
		\node [style=none] (8) at (4, 2.75) {};
		\node [style=none] (9) at (4, 0.25) {};
		\node [style=oplus] (14) at (4, 1.5) {};
		\node [style=dot] (15) at (4.5, 1.5) {};
		\node [style=none] (16) at (4, 2.5) {};
	\end{pgfonlayer}
	\begin{pgfonlayer}{edgelayer}
		\draw (4.center) to (5.center);
		\draw (3) to (2);
		\draw (15) to (14);
		\draw (9.center) to (14);
		\draw (1) to (15);
		\draw [in=-90, out=90] (15) to (16.center);
		\draw [in=-90, out=90] (14) to (0);
		\draw (8.center) to (16.center);
	\end{pgfonlayer}
\end{tikzpicture}
=
\begin{tikzpicture}
	\begin{pgfonlayer}{nodelayer}
		\node [style=X] (0) at (4, 2) {};
		\node [style=X] (1) at (4.5, 0.5) {};
		\node [style=oplus] (2) at (4.5, 1) {};
		\node [style=dot] (3) at (3.5, 1) {};
		\node [style=none] (4) at (3.5, 2.75) {};
		\node [style=none] (5) at (3.5, 0.25) {};
		\node [style=none] (8) at (4, 2.75) {};
		\node [style=none] (9) at (4, 0.25) {};
		\node [style=oplus] (14) at (4, 1.5) {};
		\node [style=dot] (15) at (4.5, 1.5) {};
		\node [style=none] (16) at (4.5, 2) {};
	\end{pgfonlayer}
	\begin{pgfonlayer}{edgelayer}
		\draw (4.center) to (5.center);
		\draw (3) to (2);
		\draw (15) to (14);
		\draw (9.center) to (14);
		\draw (1) to (15);
		\draw [in=-90, out=90] (14) to (0);
		\draw (15) to (16.center);
		\draw [in=-90, out=90, looseness=0.75] (16.center) to (8.center);
	\end{pgfonlayer}
\end{tikzpicture}
=
\begin{tikzpicture}
	\begin{pgfonlayer}{nodelayer}
		\node [style=X] (17) at (6, 2.5) {};
		\node [style=X] (18) at (6.5, 1) {};
		\node [style=oplus] (19) at (6.5, 2) {};
		\node [style=dot] (20) at (5.5, 2) {};
		\node [style=none] (21) at (5.5, 3.25) {};
		\node [style=none] (22) at (5.5, 0.25) {};
		\node [style=none] (23) at (6, 3.25) {};
		\node [style=none] (24) at (6, 0.25) {};
		\node [style=oplus] (25) at (6, 1.5) {};
		\node [style=dot] (26) at (6.5, 1.5) {};
		\node [style=none] (27) at (6.5, 2.5) {};
		\node [style=oplus] (28) at (6, 1) {};
		\node [style=dot] (29) at (5.5, 1) {};
	\end{pgfonlayer}
	\begin{pgfonlayer}{edgelayer}
		\draw (21.center) to (22.center);
		\draw (20) to (19);
		\draw (26) to (25);
		\draw (24.center) to (25);
		\draw (18) to (26);
		\draw [in=-90, out=90] (25) to (17);
		\draw (26) to (27.center);
		\draw [in=-90, out=90, looseness=0.75] (27.center) to (23.center);
		\draw (29) to (28);
	\end{pgfonlayer}
\end{tikzpicture}
=
\begin{tikzpicture}
	\begin{pgfonlayer}{nodelayer}
		\node [style=X] (17) at (6, 2) {};
		\node [style=X] (18) at (6.5, 1) {};
		\node [style=oplus] (19) at (6.5, 1.5) {};
		\node [style=dot] (20) at (5.5, 1.5) {};
		\node [style=none] (21) at (5.5, 2.75) {};
		\node [style=none] (22) at (5.5, 0.25) {};
		\node [style=none] (23) at (6, 2.75) {};
		\node [style=none] (24) at (6, 0.25) {};
		\node [style=none] (27) at (6.5, 2) {};
		\node [style=oplus] (28) at (6, 1) {};
		\node [style=dot] (29) at (5.5, 1) {};
	\end{pgfonlayer}
	\begin{pgfonlayer}{edgelayer}
		\draw (21.center) to (22.center);
		\draw (20) to (19);
		\draw [in=-90, out=90, looseness=0.75] (27.center) to (23.center);
		\draw (29) to (28);
		\draw (24.center) to (28);
		\draw (28) to (17);
		\draw (27.center) to (18);
	\end{pgfonlayer}
\end{tikzpicture}
=
\begin{tikzpicture}
	\begin{pgfonlayer}{nodelayer}
		\node [style=X] (17) at (6, 1.5) {};
		\node [style=X] (18) at (6, 2) {};
		\node [style=oplus] (19) at (6, 2.5) {};
		\node [style=dot] (20) at (5.5, 2.5) {};
		\node [style=none] (21) at (5.5, 3) {};
		\node [style=none] (22) at (5.5, 0.5) {};
		\node [style=none] (24) at (6, 0.5) {};
		\node [style=none] (27) at (6, 3) {};
		\node [style=oplus] (28) at (6, 1) {};
		\node [style=dot] (29) at (5.5, 1) {};
	\end{pgfonlayer}
	\begin{pgfonlayer}{edgelayer}
		\draw (21.center) to (22.center);
		\draw (20) to (19);
		\draw (29) to (28);
		\draw (24.center) to (28);
		\draw (28) to (17);
		\draw (27.center) to (18);
	\end{pgfonlayer}
\end{tikzpicture}
$$

The inductive case is essentially the same as the base case for 2.

\end{remark}

\subsubsection{Proof of Lemma \ref{lem:parcb}}
\label{proof:lem:parcb}
Recall the statement of the result:

{\bf Lemma  \ref{lem:parcb}: }
{\it $\Par(\cb_2)$ is a presentation for the prop $(\Par(\Mat(\F_2),+))$}

\begin{proof}
One must show that the following diagram commutes:

%
%
%

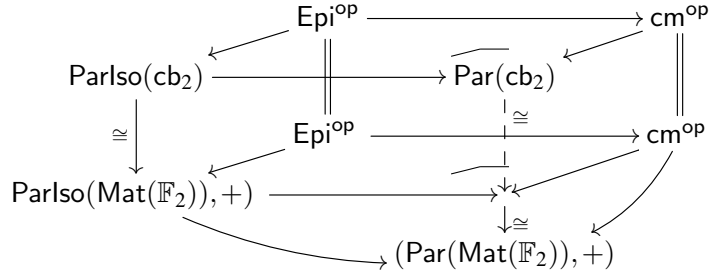
\renewcommand{\cubetopbl}{$\surj^\op$}
\renewcommand{\cubetopbr}{$\cm^\op$}
\renewcommand{\cubetopfl}{$\ParIso(\cb_2)$}
\renewcommand{\cubetopfr}{$\Par(\cb_2)$}
\renewcommand{\cubebotbl}{$\surj^\op$ }
\renewcommand{\cubebotbr}{$\cm^\op$ }
\renewcommand{\cubebotfl}{$\ParIso(\Mat(\F_2)),+)$ }
\renewcommand{\cubebotfr}{}

$$
\xymatrixrowsep{2mm}\xymatrixcolsep{2mm}
\xymatrix{
                                       & \mbox{\cubetopbl} \ar[rr] \ar[dl] \ar@{=}[dd]     &                                                  & \mbox{\cubetopbr} \ar@{=}[dd] \ar[dl] \\
\mbox{\cubetopfl} \ar[rr]  \ar[dd]_{\cong}           &                                                                                              &\mbox{\cubetopfr} \ar@{-->}[dd]^(.35){\cong}   \skewpullbackcorner[ul]              \\
                                       &  \mbox{\cubebotbl} \ar[dl] \ar[rr]                    &                                                  & \mbox{\cubebotbr} \ar@/^1pc/[ddl] \ar[dl] \\
\mbox{\cubebotfl} \ar@/_1pc/[drr] \ar[rr]  &                                                                                             & \mbox{\cubebotfr} \skewpullbackcorner[ul]    \ar@{-->}[d]^{\cong}  \\
                                                   &                                                                                             & (\Par(\Mat(\F_2)),+)
}
$$

It doesn't take to much work to show that $\ParIso(\cb_2)\cong\ParIso(\Mat(\F_2))$ is a discrete inverse category (defined in \cite[\S 4.3]{giles}).
We know that the counital completion of a discrete inverse category is the same as its Cartesian completion from \cite[Lem. 3.5]{zxa}; moreover, the Cartesian completion of  $\ParIso(\Mat(\F_2))$ is $\Par(\Mat(\F_2))$.  So this diagram commutes as a consequence.

\end{proof}

\subsubsection{Proof of Lemma \ref{lem:spancb}}
\label{proof:lem:spancb}
Recall the statement of the result:

{\bf Lemma \ref{lem:spancb}: }
{\it $\Span(\cb_2)$ is a presentation for the prop $(\Span(\Mat(\F_2)), +)$.}

\begin{proof}

\renewcommand{\cubetopbl}{$\inj(\cb_2)^\op \otimes_{\Iso(\cb_2)} \inj(\cb_2)$}
\renewcommand{\cubetopbr}{$\Par(\cb_2)$}
\renewcommand{\cubetopfl}{$\Par(\cb_2)^\op$}
\renewcommand{\cubetopfr}{$\Span(\cb_2)$}
\renewcommand{\cubebotbl}{$(\Par\Iso(\Mat(\F_2)),+)$ }
\renewcommand{\cubebotbr}{$(\Par(\Mat(\F_2)),+)$ }
\renewcommand{\cubebotfl}{$(\Par(\Mat(\F_2)),+)^\op$ }
\renewcommand{\cubebotfr}{}

$$
\xymatrixrowsep{2mm}\xymatrixcolsep{1mm}
\xymatrix{
                                       & \mbox{\cubetopbl} \ar[rr] \ar[dl] \ar[dd]^(.7){\cong}      &                                                  & \mbox{\cubetopbr}  \ar[dd]^{\cong} \ar[dl] \\
\mbox{\cubetopfl} \ar[rr]  \ar[dd]_{\cong}           &                                                                                              &\mbox{\cubetopfr} \ar@{-->}[dd]^(.35){\cong}   \skewpullbackcorner[ul]              \\
                                       &  \mbox{\cubebotbl} \ar[dl] \ar[rr]                    &                                                  & \mbox{\cubebotbr} \ar@/^1pc/[ddl] \ar[dl] \\
\mbox{\cubebotfl} \ar@/_1pc/[drr] \ar[rr]  &                                                                                             & \mbox{\cubebotfr} \skewpullbackcorner[ul]    \ar@{-->}[d]^{\cong}_F \\
                                                   &                                                                                             & (\Span(\Mat(\F_2)),+)
}
$$

The cube easily commutes.  What remains to be shown is that the universal map $F$ is an isomorphism of props.  It is clearly the identity on objects, so we just need to show it is full and faithful.

It is clearly full as any span $ n \xleftarrow{ f}  k \xrightarrow{g } m$, we have:
$$
F\left( (n \xleftarrow{f} k = k);(k = k \xrightarrow{g} m) \right)=n \xleftarrow{ f}  k \xrightarrow{g } m
$$ 
For faithfulness, we must observe given for any two isomorphic maps in $\Span(\Mat(\F_2))$:
$$
\xymatrixrowsep{2mm}\xymatrixcolsep{6mm}
\xymatrix{
          & k \ar[dl]_{f'} \ar[dd]_{\cong}^{h} \ar[dr]^{g'}\\
n  &                                                                                                    & m\\
         & k \ar[ul]^{f} \ar[ur]_{g}\\
}
$$
Then in the domain of $F$, we have:
{
\xymatrixrowsep{0mm}\xymatrixcolsep{1.7mm}
\begin{align*}
&
\xymatrix{
   & k \ar[dl]_f \ar@{=}[dr]\\
n &                                      &k
};
\xymatrix{
   & k \ar[dr]^g \ar@{=}[dl]\\
k &                                      &m
}
 =
\xymatrix{
   & k \ar[dl]_f \ar@{=}[dr]\\
n &                                      &k
};
\xymatrix{
   & k \ar@{=}[dl] \ar@{=}[dr]\\
k &                                             & k\\
   & k \ar[ul]^h \ar[ur]_h \ar[uu]^\cong_h
};
\xymatrix{
   & k \ar[dr]^g \ar@{=}[dl]\\
k &                                      &m
}\\
 &=
\xymatrix{
   & k \ar[dl]_f \ar@{=}[dr]\\
n &                                      &k
};
\xymatrix{
   & k \ar[dl]_h \ar@{=}[dr]\\
k &                                         & k
};
\xymatrix{
   & k \ar[dr]^h \ar@{=}[dl]\\
k &                                         & k
};
\xymatrix{
   & k \ar[dr]^g \ar@{=}[dl]\\
k &                                      &m
}
=
\xymatrix{
            &                                                        &k \ar[dl]_{h} \ar@{=}[dr] \ar@/_1.2pc/[ddll]_{f'}\\
            & k \ar@{=}[dr] \ar[dl]^{f}&                                                          & k \ar@{=}[dr] \ar[dl]_{h}\\
n &                                                         & k                                             &                                                         &k
};
\xymatrix{
            &                                                        & k \ar[dr]^{h} \ar@{=}[dl] \ar@/^1.2pc/[ddrr]^{g'}  \\
            & k \ar[dr]^{h}   \ar@{=}[dl] &                                                          & k \ar@{=}[dl] \ar[dr]_{g}\\
k &                                                         & k                                             &                                                         &m
}
\end{align*}
}
 
\end{proof}

\subsection{Section \ref{sec:two}}

\subsubsection{Proof of Lemma \ref{lem:injaffcb}}
\label{proof:lem:injaffcb}

Recall the statement of the result:

{\bf Lemma \ref{lem:injaffcb}:}
{\it $\inj(\Aff\cb_2)$ is a presentation for the prop $(\inj(\Aff\Mat(\F_2)),+)$.}

\begin{proof}
Consider the following diagram:

\renewcommand{\cubetopbl}{$\Iso(\cb_2)$}
\renewcommand{\cubetopbr}{$\Iso(\Aff\cb_2)$}
\renewcommand{\cubetopfl}{$\inj(\cb_2)$}
\renewcommand{\cubetopfr}{$\inj(\Aff\cb_2)$}
\renewcommand{\cubebotbl}{$(\Iso(\Mat(\F_2)),+)$ }
\renewcommand{\cubebotbr}{$(\Iso(\Aff\Mat(\F_2)),+)$ }
\renewcommand{\cubebotfl}{$(\inj(\Mat(\F_2)),+)$ }
\renewcommand{\cubebotfr}{}

$$
\xymatrixrowsep{2mm}\xymatrixcolsep{1.5mm}
\xymatrix{
                                       & \mbox{\cubetopbl} \ar[rr] \ar[dl] \ar[dd]^(.7){\cong}      &                                                  & \mbox{\cubetopbr}  \ar[dd]^{\cong} \ar[dl] \\
\mbox{\cubetopfl} \ar[rr]  \ar[dd]_{\cong}           &                                                                                              &\mbox{\cubetopfr} \ar@{-->}[dd]^(.35){\cong}   \skewpullbackcorner[ul]              \\
                                       &  \mbox{\cubebotbl} \ar[dl] \ar[rr]                    &                                                  & \mbox{\cubebotbr} \ar@/^1pc/[ddl] \ar[dl] \\
\mbox{\cubebotfl} \ar@/_1pc/[drr] \ar[rr]  &                                                                                             & \mbox{\cubebotfr} \skewpullbackcorner[ul]    \ar@{-->}[d]^{\cong}_F  \\
                                                   &                                                                                             & (\inj(\Aff\Mat(\F_2)),+)
}
$$

 The rear and left faces of the cube commute and the vertical maps are all isomorphisms. Therefore, the whole cube commutes via universal property of the pushout, with the upper universal map being an isomorphism.

We seek to show that the lower universal map  $F$ is also an isomorphism.  It is clearly the identity on objects, so we just have to show fullness and faithfulness.

For fullness, consider any map $n\xrightarrowtail{(A,x)} m$ in $(\inj(\Aff\Mat(\F_2)),+)$.  Note that this can be factored into:
$$
n\xrightarrowtail{(A,0)} m \xrightarrowiso{(1,x)}  m
$$
Which lies in the image of $F$ as $m \xrightarrowiso{(1,x)} m$ is an isomorphism.

For faithfulness, we show that every map in $(\Iso(\Aff\Mat(\F_2)),+)+_{(\Iso(\Mat(\F_2)),+)} (\inj(\Mat(\F_2)),+)$ can be factored uniquely in this way. 
There are two cases:
$$
\left( n \xrightarrowtail{ A} m ; m \xrightarrowiso{(B, x)} m \right)
= \left( n \xrightarrowtail{ A} m ; m \xrightarrowiso{(B, 0)} m; m \xrightarrowiso{(1, x)}  m \right)
= \left( n \xrightarrowtail{ A;B}  m\xrightarrowiso{(1, x)}  m \right)
$$
$$
\left(n \xrightarrowtail{ (A,x)} m ; m \xrightarrowiso{B} m \right)
= \left( n \xrightarrowtail{ (A,0)}m; m \xrightarrowiso{(1,x)} m ; m \xrightarrowiso{B} m \right)
= \left( n \xrightarrowtail{ A }m; m \xrightarrowiso{(B,B(x))} m  \right)
= \left( n \xrightarrowtail{ A;B }m; m \xrightarrowiso{(1,B(x))} m  \right)
$$
So every map in this pushout has the correct form, which is unique by construction.
\end{proof}

\subsubsection{Remark on Definition \ref{def:parisoaffcb}}

\begin{remark}
\label{rem:parisoaffcb}
$ (\inj(\Aff\cb_2)^{+1})^\op \otimes_{\Iso(\Aff\cb_2)^{+1}}  \inj(\Aff\cb_2)^{+1}$ is actually a distributive law because the only only nontrivial situation arises when controlled-not gates are sandwiched between black, or black $\pi$ units/counits on their target wires.  The case where there are no controlled not gates in between is resolved by the new axiom we have added.  When there are more controlled-not gates, they can be pushed past each other as follows:
$$
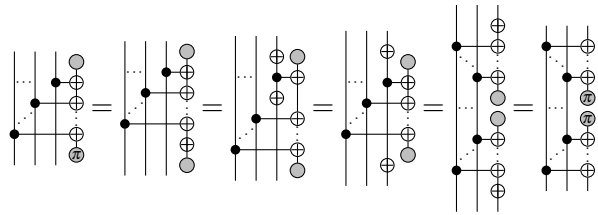
\begin{tikzpicture}
	\begin{pgfonlayer}{nodelayer}
		\node [style=X] (0) at (2.5, 0.25) {};
		\node [style=oplus] (1) at (2.5, -0.75) {};
		\node [style=oplus] (2) at (2.5, -1.5) {};
		\node [style=dot] (3) at (1.5, -0.75) {};
		\node [style=dot] (4) at (1, -1.5) {};
		\node [style=none] (5) at (1.5, 0.5) {};
		\node [style=none] (6) at (1, 0.5) {};
		\node [style=none] (7) at (1.25, -1) {$\iddots$};
		\node [style=none] (8) at (2.5, -1) {$\vdots$};
		\node [style=X] (9) at (2.5, -2) {$\pi$};
		\node [style=none] (10) at (1, -2.25) {};
		\node [style=none] (11) at (1.5, -2.25) {};
		\node [style=none] (12) at (1.25, -0.25) {$\cdots$};
		\node [style=none] (17) at (2, -2.25) {};
		\node [style=none] (18) at (2, 0.5) {};
		\node [style=oplus] (19) at (2.5, -0.25) {};
		\node [style=dot] (20) at (2, -0.25) {};
	\end{pgfonlayer}
	\begin{pgfonlayer}{edgelayer}
		\draw (0) to (1);
		\draw (1) to (3);
		\draw (5.center) to (3);
		\draw (6.center) to (4);
		\draw (4) to (2);
		\draw (4) to (10.center);
		\draw (11.center) to (3);
		\draw (9) to (2);
		\draw (17.center) to (20);
		\draw (20) to (18.center);
		\draw (19) to (20);
	\end{pgfonlayer}
\end{tikzpicture}
=
\begin{tikzpicture}
	\begin{pgfonlayer}{nodelayer}
		\node [style=X] (0) at (2.5, 0.25) {};
		\node [style=oplus] (1) at (2.5, -0.75) {};
		\node [style=oplus] (2) at (2.5, -1.5) {};
		\node [style=dot] (3) at (1.5, -0.75) {};
		\node [style=dot] (4) at (1, -1.5) {};
		\node [style=none] (5) at (1.5, 0.5) {};
		\node [style=none] (6) at (1, 0.5) {};
		\node [style=none] (7) at (1.25, -1) {$\iddots$};
		\node [style=none] (8) at (2.5, -1) {$\vdots$};
		\node [style=X] (9) at (2.5, -2.5) {};
		\node [style=none] (10) at (1, -2.75) {};
		\node [style=none] (11) at (1.5, -2.75) {};
		\node [style=none] (12) at (1.25, -0.25) {$\cdots$};
		\node [style=none] (17) at (2, -2.75) {};
		\node [style=none] (18) at (2, 0.5) {};
		\node [style=oplus] (19) at (2.5, -0.25) {};
		\node [style=dot] (20) at (2, -0.25) {};
		\node [style=oplus] (21) at (2.5, -2) {};
	\end{pgfonlayer}
	\begin{pgfonlayer}{edgelayer}
		\draw (0) to (1);
		\draw (1) to (3);
		\draw (5.center) to (3);
		\draw (6.center) to (4);
		\draw (4) to (2);
		\draw (4) to (10.center);
		\draw (11.center) to (3);
		\draw (9) to (2);
		\draw (17.center) to (20);
		\draw (20) to (18.center);
		\draw (19) to (20);
	\end{pgfonlayer}
\end{tikzpicture}
=
\begin{tikzpicture}
	\begin{pgfonlayer}{nodelayer}
		\node [style=X] (0) at (2.5, 0.75) {};
		\node [style=oplus] (1) at (2.5, -0.75) {};
		\node [style=oplus] (2) at (2.5, -1.5) {};
		\node [style=dot] (3) at (1.5, -0.75) {};
		\node [style=dot] (4) at (1, -1.5) {};
		\node [style=none] (5) at (1.5, 1.25) {};
		\node [style=none] (6) at (1, 1.25) {};
		\node [style=none] (7) at (1.25, -1) {$\iddots$};
		\node [style=none] (8) at (2.5, -1) {$\vdots$};
		\node [style=X] (9) at (2.5, -2) {};
		\node [style=none] (10) at (1, -2.25) {};
		\node [style=none] (11) at (1.5, -2.25) {};
		\node [style=none] (12) at (1.25, 0.25) {$\cdots$};
		\node [style=none] (17) at (2, -2.25) {};
		\node [style=none] (18) at (2, 1.25) {};
		\node [style=oplus] (19) at (2.5, 0.25) {};
		\node [style=dot] (20) at (2, 0.25) {};
		\node [style=oplus] (21) at (2, -0.25) {};
		\node [style=oplus] (22) at (2, 0.75) {};
	\end{pgfonlayer}
	\begin{pgfonlayer}{edgelayer}
		\draw (0) to (1);
		\draw (1) to (3);
		\draw (5.center) to (3);
		\draw (6.center) to (4);
		\draw (4) to (2);
		\draw (4) to (10.center);
		\draw (11.center) to (3);
		\draw (9) to (2);
		\draw (17.center) to (20);
		\draw (20) to (18.center);
		\draw (19) to (20);
	\end{pgfonlayer}
\end{tikzpicture}
=
\begin{tikzpicture}
	\begin{pgfonlayer}{nodelayer}
		\node [style=X] (0) at (2.5, 0.25) {};
		\node [style=oplus] (1) at (2.5, -0.75) {};
		\node [style=oplus] (2) at (2.5, -1.5) {};
		\node [style=dot] (3) at (1.5, -0.75) {};
		\node [style=dot] (4) at (1, -1.5) {};
		\node [style=none] (5) at (1.5, 1) {};
		\node [style=none] (6) at (1, 1) {};
		\node [style=none] (7) at (1.25, -1) {$\iddots$};
		\node [style=none] (8) at (2.5, -1) {$\vdots$};
		\node [style=X] (9) at (2.5, -2) {};
		\node [style=none] (10) at (1, -2.75) {};
		\node [style=none] (11) at (1.5, -2.75) {};
		\node [style=none] (12) at (1.25, -0.25) {$\cdots$};
		\node [style=none] (17) at (2, -2.75) {};
		\node [style=none] (18) at (2, 1) {};
		\node [style=oplus] (19) at (2.5, -0.25) {};
		\node [style=dot] (20) at (2, -0.25) {};
		\node [style=oplus] (21) at (2, -2.25) {};
		\node [style=oplus] (22) at (2, 0.5) {};
	\end{pgfonlayer}
	\begin{pgfonlayer}{edgelayer}
		\draw (0) to (1);
		\draw (1) to (3);
		\draw (5.center) to (3);
		\draw (6.center) to (4);
		\draw (4) to (2);
		\draw (4) to (10.center);
		\draw (11.center) to (3);
		\draw (9) to (2);
		\draw (17.center) to (20);
		\draw (20) to (18.center);
		\draw (19) to (20);
	\end{pgfonlayer}
\end{tikzpicture}
=
\begin{tikzpicture}
	\begin{pgfonlayer}{nodelayer}
		\node [style=none] (24) at (4.5, 1.75) {};
		\node [style=none] (25) at (4, 1.75) {};
		\node [style=X] (28) at (5, -1) {};
		\node [style=none] (29) at (4, -3.25) {};
		\node [style=none] (30) at (4.5, -3.25) {};
		\node [style=none] (32) at (5, -3.25) {};
		\node [style=none] (33) at (5, 1.75) {};
		\node [style=oplus] (37) at (5, 1.25) {};
		\node [style=oplus] (38) at (5, -1.5) {};
		\node [style=dot] (39) at (4.5, -1.5) {};
		\node [style=oplus] (40) at (5, -2.25) {};
		\node [style=dot] (41) at (4, -2.25) {};
		\node [style=oplus] (42) at (5, 0) {};
		\node [style=dot] (43) at (4.5, 0) {};
		\node [style=oplus] (44) at (5, 0.75) {};
		\node [style=dot] (45) at (4, 0.75) {};
		\node [style=oplus] (46) at (5, -2.75) {};
		\node [style=X] (47) at (5, -0.5) {};
		\node [style=none] (48) at (4.25, -0.75) {$\cdots$};
		\node [style=none] (49) at (4.25, -1.75) {$\iddots$};
		\node [style=none] (50) at (4.25, 0.5) {$\ddots$};
		\node [style=none] (51) at (5, 0.5) {$\vdots$};
		\node [style=none] (52) at (5, -1.75) {$\vdots$};
	\end{pgfonlayer}
	\begin{pgfonlayer}{edgelayer}
		\draw (38) to (39);
		\draw (40) to (41);
		\draw (42) to (43);
		\draw (44) to (45);
		\draw (29.center) to (25.center);
		\draw (24.center) to (30.center);
		\draw (38) to (28);
		\draw (40) to (32.center);
		\draw (47) to (42);
		\draw (44) to (33.center);
	\end{pgfonlayer}
\end{tikzpicture}
=
\begin{tikzpicture}
	\begin{pgfonlayer}{nodelayer}
		\node [style=none] (24) at (4.5, 1.25) {};
		\node [style=none] (25) at (4, 1.25) {};
		\node [style=X] (28) at (5, -1) {$\pi$};
		\node [style=none] (29) at (4, -2.75) {};
		\node [style=none] (30) at (4.5, -2.75) {};
		\node [style=none] (32) at (5, -2.75) {};
		\node [style=none] (33) at (5, 1.25) {};
		\node [style=oplus] (38) at (5, -1.5) {};
		\node [style=dot] (39) at (4.5, -1.5) {};
		\node [style=oplus] (40) at (5, -2.25) {};
		\node [style=dot] (41) at (4, -2.25) {};
		\node [style=oplus] (42) at (5, 0) {};
		\node [style=dot] (43) at (4.5, 0) {};
		\node [style=oplus] (44) at (5, 0.75) {};
		\node [style=dot] (45) at (4, 0.75) {};
		\node [style=X] (47) at (5, -0.5) {$\pi$};
		\node [style=none] (48) at (4.25, -0.75) {$\cdots$};
		\node [style=none] (49) at (4.25, -1.75) {$\iddots$};
		\node [style=none] (50) at (4.25, 0.5) {$\ddots$};
		\node [style=none] (51) at (5, 0.5) {$\vdots$};
		\node [style=none] (52) at (5, -1.75) {$\vdots$};
	\end{pgfonlayer}
	\begin{pgfonlayer}{edgelayer}
		\draw (38) to (39);
		\draw (40) to (41);
		\draw (42) to (43);
		\draw (44) to (45);
		\draw (29.center) to (25.center);
		\draw (24.center) to (30.center);
		\draw (38) to (28);
		\draw (40) to (32.center);
		\draw (47) to (42);
		\draw (44) to (33.center);
	\end{pgfonlayer}
\end{tikzpicture}
$$
\end{remark}

\subsubsection{Proof of Lemma \ref{lem:paraffcb}}
\label{proof:lem:paraffcb}

Recall the statement of the result:

{\bf Lemma \ref{lem:paraffcb}: }
{\it  $\pr\Aff\cb_2$ is a presentation for the prop $(\Par(\Aff\Vect(\F_2))^*,+)$.}

\begin{proof}
%

\renewcommand{\cubetopbl}{$\surj^\op$}
\renewcommand{\cubetopbr}{$\cm^\op$}
\renewcommand{\cubetopfl}{$\pr\iso\Aff\cb_2$}
\renewcommand{\cubetopfr}{$\pr\Aff\cb_2$}
\renewcommand{\cubebotbl}{$\surj^\op$ }
\renewcommand{\cubebotbr}{$\cm^\op$ }
\renewcommand{\cubebotfl}{$(\ParIso(\Aff\Vect(\F_2))^*,+)$ }
\renewcommand{\cubebotfr}{}

$$
\xymatrixrowsep{2mm}\xymatrixcolsep{1mm}
\xymatrix{
                                       & \mbox{\cubetopbl} \ar[rr] \ar[dl] \ar@{=}[dd]     &                                                  & \mbox{\cubetopbr} \ar@{=}[dd] \ar[dl] \\
\mbox{\cubetopfl} \ar[rr]  \ar[dd]_{\cong}           &                                                                                              &\mbox{\cubetopfr} \ar@{-->}[dd]^(.35){\cong}   \skewpullbackcorner[ul]              \\
                                       &  \mbox{\cubebotbl} \ar[dl] \ar[rr]                    &                                                  & \mbox{\cubebotbr} \ar@/^1pc/[ddl] \ar[dl] \\
\mbox{\cubebotfl} \ar@/_1pc/[drr] \ar[rr]  &                                                                                             & \mbox{\cubebotfr} \skewpullbackcorner[ul]    \ar@{-->}[d]_F^{\cong}  \\
                                                   &                                                                                             & (\Par(\Aff\Vect(\F_2))^*,+) 
}
$$

We know that $\pr\iso\Aff\cb_2\cong \ParIso(\Aff\Vect(\F_2))^*,+)$ is a discrete inverse category by \cite[Prop. 3.4]{cnot}.

The cube commutes by the universal property of the pushout, as before.

We just have to show that the universal map $F$ is an isomorphism.  It is clearly the identity on objects, so we just have to show it is full and faithful.
This follows from essentially the same argument as in the linear case.

\end{proof}

\subsubsection{Proof of Lemma \ref{lem:spanaffcb}}
\label{proof:lem:spanaffcb}

Recall the statement of the result:

{\bf Lemma \ref{lem:spanaffcb}: }
{\it  $\sp\Aff\cb_2$ is a presentation for the prop $(\Span(\Aff\Vect(\F_2))^*,+)$.}

\begin{proof}
\renewcommand{\cubetopbl}{$\pr\iso\Aff\cb_2$}
\renewcommand{\cubetopbr}{$\pr\Aff\cb_2$}
\renewcommand{\cubetopfl}{$\pr\Aff\cb_2^\op$}
\renewcommand{\cubetopfr}{$\sp\Aff\cb_2$}
\renewcommand{\cubebotbl}{$(\ParIso(\Aff\Vect(\F_2))^*,+)$ }
\renewcommand{\cubebotbr}{$(\Par(\Aff\Vect(\F_2))^*,+)$ }
\renewcommand{\cubebotfl}{$(\Par(\Aff\Vect(\F_2))^*,+)^\op$ }
\renewcommand{\cubebotfr}{}

$$
\hspace*{-1cm}
\xymatrixrowsep{2mm}\xymatrixcolsep{.5mm}
\xymatrix{
                                       & \mbox{\cubetopbl} \ar[rr] \ar[dl] \ar[dd]^(.7){\cong}      &                                                  & \mbox{\cubetopbr}  \ar[dd]^{\cong} \ar[dl] \\
\mbox{\cubetopfl} \ar[rr]  \ar[dd]_{\cong}           &                                                                                              &\mbox{\cubetopfr} \ar@{-->}[dd]^(.35){\cong}   \skewpullbackcorner[ul]              \\
                                       &  \mbox{\cubebotbl} \ar[dl] \ar[rr]                    &                                                  & \mbox{\cubebotbr} \ar@/^1pc/[ddl] \ar[dl] \\
\mbox{\cubebotfl} \ar@/_1pc/[drr] \ar[rr]  &                                                                                             & \mbox{\cubebotfr} \skewpullbackcorner[ul]    \ar@{-->}[d]_F^{\cong}  \\
                                                   &                                                                                             & (\Span(\Aff\Vect(\F_2))^*,+)
}
$$

 The rear and left faces of the cube commute and the vertical maps are all isomorphisms. Therefore, the whole cube commutes by the universal property of the pushout, with the upper universal map being an isomorphism.

We seek to show that the lower universal map  $F$ is also an isomorphism.  It is clearly the identity on objects, so we just have to show fullness and faithfulness.

For fullness, let us first consider the nonempty case; that is a map $\F_2^n \xleftarrow{(A,x)} \F_2^k \xrightarrow{(B,y)}\F^m$ in $(\Span(\Aff\Vect(\F_2))^*,+)$.  This is in the image of the following diagram under $F$:
$$
(\F_2^n \xleftarrow{(A,x)} \F_2^k  = \F_2^k); (\F_2^k = \F_2^k  \xrightarrow{(B,y)}\F^m)
$$ 
Otherwise, consider a map of the form  $\F_2^n \xleftarrow{?} \emptyset  \xrightarrow{?}\F^m$.  This the image of the following diagram:
$$
(\F_2^n \xleftarrow{?} \emptyset \xrightarrow {?} \F_2^0  );(\F_2^0 \xleftarrow{?} \emptyset  \xrightarrow{?}\F^m)
$$
For faithfulness, again, we separate the proof into two cases.  The functor is faithful on diagrams in $(\Span(\Aff\Vect(\F_2))^*,+)$ with nonempty apex by the same argument as in Lemma \ref{lem:spancb}.
The case for spans with empty apex follows immediately as the only endomorphism on the empty set is the identity; thus,  isomorphic spans must be equal on the nose.

\end{proof}

\subsection{Section \ref{sec:three}}

\subsubsection{Proof of Lemma \ref{lem:parand}}
\label{proof:lem:parand}
Recall the statement of the result:

{\bf Lemma \ref{lem:parand}: }
{\it $\pr\f_2$ is a presentation for the the full subcategory $(\FPar_2,\times)$ of $(\Par(\FSets),\times)$ with objects powers of two.}

\begin{proof}
One has to show that the following diagram commutes:

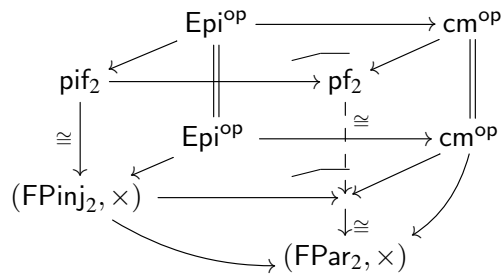
\renewcommand{\cubetopbl}{$\surj^\op$}
\renewcommand{\cubetopbr}{$\cm^\op$}
\renewcommand{\cubetopfl}{$\pr\iso\f_2$}
\renewcommand{\cubetopfr}{$\pr\f_2$}
\renewcommand{\cubebotbl}{$\surj^\op$ }
\renewcommand{\cubebotbr}{$\cm^\op$ }
\renewcommand{\cubebotfl}{$(\FPinj_2,\times)$ }
\renewcommand{\cubebotfr}{}

$$
\xymatrixrowsep{2mm}\xymatrixcolsep{2mm}
\xymatrix{
                                       & \mbox{\cubetopbl} \ar[rr] \ar[dl] \ar@{=}[dd]     &                                                  & \mbox{\cubetopbr} \ar@{=}[dd] \ar[dl] \\
\mbox{\cubetopfl} \ar[rr]  \ar[dd]_{\cong}           &                                                                                              &\mbox{\cubetopfr} \ar@{-->}[dd]^(.35){\cong}   \skewpullbackcorner[ul]              \\
                                       &  \mbox{\cubebotbl} \ar[dl] \ar[rr]                    &                                                  & \mbox{\cubebotbr} \ar@/^1pc/[ddl] \ar[dl] \\
\mbox{\cubebotfl} \ar@/_1pc/[drr] \ar[rr]  &                                                                                             & \mbox{\cubebotfr} \skewpullbackcorner[ul]    \ar@{-->}[d]^{\cong}  \\
                                                   &                                                                                             & (\FPar_2,\times)
}
$$

Again, the proof is essentially the same as for the linear and affine cases; the only difference being that the Cartesian completion of $\FPinj_2$ is $\FPar_2$.

\end{proof}

\subsubsection{Proof of Lemma \ref{lem:spanand}}
\label{proof:lem:spanand}
Recall the statement of the result:

{\bf Lemma \ref{lem:spanand}: }
{\sf $\sp\f_2$ is a presentation for  the full subcategory $(\FSpan_2,\times)$ of $(\Span(\FSets),\times)$ with objects powers of two.}

\begin{proof}
One has to show that the following diagram commutes:

\renewcommand{\cubetopbl}{$\pr\iso\f_2$}
\renewcommand{\cubetopbr}{$ \pr\f_2$}
\renewcommand{\cubetopfl}{$\pr\f_2^\op$}
\renewcommand{\cubetopfr}{$\sp\f_2$}
\renewcommand{\cubebotbl}{$(\FPinj_2,\times)$ }
\renewcommand{\cubebotbr}{$(\FPar_2,\times)$ }
\renewcommand{\cubebotfl}{$(\FPar_2,\times)^\op$ }
\renewcommand{\cubebotfr}{}

$$
\xymatrixrowsep{2mm}\xymatrixcolsep{0mm}
\xymatrix{
                                       & \mbox{\cubetopbl} \ar[rr] \ar[dl] \ar[dd]^(.7){\cong}      &                                                  & \mbox{\cubetopbr}  \ar[dd]^{\cong} \ar[dl] \\
\mbox{\cubetopfl} \ar[rr]  \ar[dd]_{\cong}           &                                                                                              &\mbox{\cubetopfr} \ar@{-->}[dd]^(.35){\cong}   \skewpullbackcorner[ul]              \\
                                       &  \mbox{\cubebotbl} \ar[dl] \ar[rr]                    &                                                  & \mbox{\cubebotbr} \ar@/^1pc/[ddl] \ar[dl] \\
\mbox{\cubebotfl} \ar@/_1pc/[drr] \ar[rr]  &                                                                                             & \mbox{\cubebotfr} \skewpullbackcorner[ul]    \ar@{-->}[d]^{\cong}  \\
                                                   &                                                                                             &( \FSpan_2,\times)
}
$$

This follows from \cite[Lem. 4.3]{zxa}.

\end{proof}

\end{document}